\documentclass[a4paper,reqno,11pt]{amsart}

\usepackage{amsmath,amsthm,amssymb,dsfont,graphicx,xspace,xcolor}
\usepackage[T1]{fontenc}
\usepackage{thm-restate}
\usepackage[hidelinks]{hyperref}
\usepackage{bm}
\usepackage{mathrsfs}
\usepackage[shortlabels]{enumitem}
\usepackage{float}
\usepackage{pict2e}
\usepackage{mathtools}
\usepackage{tikz}
\usepackage{xstring,refcount}
\usepackage[longnamesfirst,numbers,sort&compress]{natbib}
\usepackage{cleveref}

\renewenvironment{itemize}{\begin{itemorig}[noitemsep,leftmargin=3em]}{\end{itemorig}}

\renewenvironment{enumerate}{\begin{enumorig}[label=\textup{(\roman*)}]}{\end{enumorig}}
\newenvironment{enumerateNum}{\begin{enumorig}[label=\textup{(\arabic*)}]}{\end{enumorig}}
\newenvironment{enumerateAlpha}{\begin{enumorig}[label=\textup{(\alph*)}]}{\end{enumorig}}

\theoremstyle{plain}
\newtheorem{thm}{Theorem}
\newtheorem*{thm*}{Theorem}
\newtheorem{theorem}[thm]{Theorem}

\newtheorem{lemma}[thm]{Lemma}
\newtheorem*{lemma*}{Lemma}
\newtheorem{corollary}[thm]{Corollary}
\newtheorem*{cor*}{Corollary}

\newtheorem{claim}[thm]{Claim}
\newtheorem{subclaim}{Subclaim}[claim]

\theoremstyle{remark}

\crefname{obs}{Observation}{Observations}
\newtheorem*{lem*}{Lemma}

\theoremstyle{definition}
\newtheorem{conjecture}[thm]{Conjecture}
\newtheorem{problem}[thm]{Problem}
\newtheorem*{conj*}{Conjecture}
\crefname{lem}{Lemma}{Lemmas}
\crefname{thm}{Theorem}{Theorems}
\crefname{cor}{Corollary}{Corollaries}

\newenvironment{proofclaim}[1][Proof of claim]
  {\begin{proof}[#1]}{\end{proof}} 

\newenvironment{proofsubclaim}[1][Proof of subclaim]
    {\begin{proof}[#1]}{\end{proof}}

\DeclareMathOperator{\dom}{Dom}
\DeclareMathOperator{\e}{e}

\newcommand{\cCAL}{\mathcal{C}}
\newcommand{\iCAL}{\mathcal{I}}
\newcommand{\tCAL}{\mathcal{T}}
\newcommand{\kCAL}{\mathcal{K}}
\newcommand{\xCAL}{\mathcal{X}}
\newcommand{\yCAL}{\mathcal{Y}}
\newcommand{\bSCR}{\mathscr{B}}
\newcommand{\cSCR}{\mathscr{C}}
\newcommand{\lSCR}{\mathscr{L}}
\newcommand{\zSCR}{\mathscr{Z}}

\newcommand{\dic}{\vec{\chi}}

\newcommand{\bidbis}[1]{\overset{\text{\tiny$\bm\leftrightarrow$}}{#1}}
\newcommand{\bic}{\bidbis{\omega}}
\newcommand{\delmax}{\Delta_{\max}}

\newcommand{\delm}{\Delta_{\rm m}}
\newcommand{\dm}{d_{\rm m}}
\newcommand{\deltil}{\Tilde{\Delta}}
\newcommand{\delplus}{\Delta^{\!+}}

\newcommand{\EE}{\mathbb{E}}
\newcommand{\PP}{\mathbb{P}}
\newcommand{\DeltaDecomp}{\Delta_{\rm DDL}}

\tikzset{vertex/.style = {circle,fill=black,minimum size=5pt, inner sep=0pt, outer sep=2pt}}
\tikzset{tedge/.style = {very thick}}
\tikzset{digon/.style = {thick,<->,> = latex}}
\tikzset{region/.style={rounded corners, draw, very thick, outer sep=2pt, inner sep=10pt}}
\usetikzlibrary{fit}

\let\le\leqslant

\let\leq\leqslant
\let\geq\geqslant

\renewcommand{\epsilon}{\varepsilon}
\newcommand{\epsilonb}{\epsilon_b}
\renewcommand{\emptyset}{\varnothing}

\newcommand{\fb}{13b}

\newcommand{\add}{{\rm add}}
\newcommand{\del}{{\rm del}}
\newcommand{\Zsp}{Z_s^{\add}}
\newcommand{\Zsm}{Z_s^{\del}}

\hypersetup{
    pdftitle={coloring digraphs with \texorpdfstring{$\deltil-b$}{Delta - b} colors},
    colorlinks,
    linkcolor={red!50!black},
    citecolor={blue!50!black},
    urlcolor={blue!80!black}
}

\title{Coloring digraphs with $\deltil-b$ colors}

\author[K. Kawarabayashi]{Ken-ichi Kawarabayashi}
\address{National Institute of Informatics, The University of Tokyo, Japan}
\email{k\_keniti@nii.ac.jp}

\author[L. Picasarri-Arrieta]{Lucas Picasarri-Arrieta}
\address{National Institute of Informatics, Japan}
\email{lpicasarr@nii.ac.jp}

\thanks{Research supported by JSPS KAKENHI JP20A402 and 22H05001 and by JST ASPIRE JPMJAP2302}

\makeatletter
\let\old@setaddresses\@setaddresses
\def\@setaddresses{%
  \bigskip
  \bgroup
  \parindent 0pt
  \old@setaddresses
  \egroup
}
\makeatother

\begin{document}

\begin{abstract} 
The {\it dichromatic number} of a digraph is the minimum number of colors needed to partition its vertex set into acyclic subdigraphs.
A {\it biclique} is a set of vertices inducing all possible pairs of opposite arcs. For a digraph $D$, define $\deltil(D) = \max_{v\in V(D)} \sqrt{d^+(v) \cdot d^-(v)}$. 

We prove that, for every fixed integer $b\in\mathbb{N}$, every digraph $D$ with $\deltil(D) = \Delta$ being sufficiently large with respect to $b$ either contains a biclique whose size exceeds $\Delta-2b$ or has dichromatic number at most $\Delta-b$. 

This extends a classical result of Reed to the directed setting and supports a conjecture of the present authors. 
Furthermore, the theorem is tight, as for all integers $b$ and $\Delta\geq 3b$ there exists a digraph $D$ with $\deltil(D)= \Delta$, dichromatic number $\Delta-b+1$, and whose largest biclique has size $\Delta-2b+1$.
\end{abstract}

\maketitle

\clearpage

\section{Introduction}

It is well-known that the chromatic number $\chi(G)$ of a graph $G$ is bounded above by $\Delta(G)+1$, where $\Delta(G)$ denotes the maximum degree of $G$.
This bound is achieved by complete graphs, which motivates the study of the relationships between $\chi(G)$, $\Delta(G)$, and the clique number $\omega(G)$ of a graph~$G$.
As a first step in this direction, Brooks~\cite{brooksMPCPS37} obtained a seminal result stating that every graph $G$ satisfies
\[
    \chi(G) \leq \max \{\Delta(G), \omega(G)\},
\]
unless $\Delta(G)=2$ and $G$ contains an odd cycle.
Pushing this direction further, in 1977 Borodin and Kostochka~\cite{borodinJCTB23} posed a celebrated conjecture stating that
\[
    \chi(G) \leq \max \{\Delta(G)-1, \omega(G)\}
\]
holds for every graph $G$ with $\Delta(G)\geq 9$.
Even though the conjecture is still open in general, Reed~\cite{reedJCTB76} proved that it holds for graphs $G$ with $\Delta(G)\geq 10^{10}$.
Moreover, in 1998 Reed~\cite{reedJGT27} conjectured that these inequalities are just the tip of the iceberg, and posed the following famous conjecture.

\begin{conjecture}[\cite{reedJGT27}]
    \label{conj:reed}
    Every graph $G$ satisfies $\chi(G) \leq \lceil \frac{1}{2}(\Delta(G)+1 + \omega(G))\rceil$.
\end{conjecture}
As supporting evidence for the conjecture, in the same paper Reed proved that, for every graph $G$,
$\chi(G) \leq \tfrac{1}{2}\big(\Delta(G) + 1 + \omega(G)\big)$
holds whenever $\Delta(G)$ is sufficiently large and $\omega(G)$ is sufficiently close to $\Delta(G)$; formally when $\Delta(G) \geq \Delta_0$ and $\omega(G) \geq (1-\epsilon)\Delta(G)$ for some absolute constants $\Delta_0\in \mathbb{N}$ and $\epsilon\in (0,1)$.
This has the following two main consequences.

\begin{corollary}[\cite{reedJGT27}]
    \label{cor:reed_1}
    There exists $\epsilon>0$ such that every graph $G$ satisfies 
    \[
        \chi(G) \leq \lceil (1-\epsilon)(\Delta(G)+1) + \epsilon \omega(G)\rceil.
    \]
\end{corollary}

\begin{corollary}[\cite{reedJGT27}]
    \label{cor:reed_2}
    For every $b\in \mathbb{Z}$, there exists $\Delta_b\in \mathbb{N}$ such that every graph $G$ with $\Delta(G) \geq \Delta_b$ and $\omega(G) \leq \Delta(G)-2b$ satisfies $\chi(G) \leq \Delta(G)-b$.
\end{corollary}

Observe that Conjecture~\ref{conj:reed} is precisely the statement of Corollary~\ref{cor:reed_1} for $\epsilon = 1/2$. This gave rise to a line of research aimed at determining the largest value of $\epsilon>0$ for which Corollary~\ref{cor:reed_1} holds.
The current best result is due to Hurley, Joannis de Verclos, and Kang~\cite{hurleyAC22}, who proved that it holds for $\epsilon = 0.119$ for graphs with sufficiently large maximum degree. This improves on earlier results obtained by Bonamy, Perrett, and Postle~\cite{bonamyJCTB155} and by Delcourt and Postle~\cite{delcourtENDM61}.

\medskip

Let us briefly discuss the sharpness of Corollary~\ref{cor:reed_2} and its connection to Conjecture~\ref{conj:reed}.
First, observe that the threshold $\Delta_b$ in Corollary~\ref{cor:reed_2} must depend on $b$. Indeed, for every integer $b\in \mathbb{N}$, there exists a graph $G$ with $\Delta(G) = 6b+2$, $\omega(G) = 4b+2$, and $\chi(G) = 5b+3$. One example is the graph obtained from a $5$-cycle by blowing up each vertex into a clique of size $2b+1$.
In particular, such a graph does not satisfy the conclusion of Corollary~\ref{cor:reed_2}, showing that $\Delta_b$ must be at least $6b+3$.

However, the dependence of $\Delta_b$ on $b$ might not be necessary under the stronger hypothesis $\omega(G) \leq \Delta -2b-1$. Indeed, Conjecture~\ref{conj:reed} is easily seen to be equivalent to the following.

\begin{conjecture}
    \label{conj:reed_bis}
    For every $b\in \mathbb{Z}$ and every graph $G$, if $\omega(G) \leq \Delta-2b-1$ then $\chi(G) \leq \Delta-b$.
\end{conjecture}

Under the assumption that $\Delta_b$ is arbitrarily large compared to $b$, the condition $\omega(G) \leq \Delta-2b$ of Corollary~\ref{cor:reed_2} might not be optimal, as it is known to be tight only up to an $o(b)$ term (see the proof of~\cite[Theorem~2]{reedJGT27}).

\medskip

The goal of this paper is to extend Corollary~\ref{cor:reed_2} to the directed setting, building on earlier work extending the aforementioned results.
To this end, let us define three digraph parameters that respectively extend $\chi$, $\Delta$, and $\omega$  to the directed setting. Here, an {\it extension} of a graph parameter is understood in the following sense: if $D$ is a symmetric digraph (that is, $uv$ is an arc of $D$ if and only if $vu$ is), the value of the digraph parameter on $D$ coincides with the value of the corresponding graph parameter on the underlying graph of $D$. This is indeed the case for $\dic$, $\deltil$, and $\bic$ defined as follows.

Let $D$ be a digraph. We let $\deltil(D)$ denote the maximum geometric mean of the in- and out-degrees of the vertices of $D$, that is $\deltil(D) = \max \{ \sqrt{d^-(v)\cdot d^+(v)} : v \in V(D) \}$.
A \textit{biclique} of $D$ is a set of vertices inducing a complete digraph, which is a digraph containing all possible arcs.
The \textit{biclique number} $\bic(D)$ of $D$ is the size of the largest biclique of $D$.
The \textit{dichromatic number} of $D$, denoted by $\dic(D)$, is the smallest $k\in \mathbb{N}$ for which $D$ admits a partition $V_1,\dots,V_k$ of its vertex-set such that $D[V_i]$ is acyclic for every $i\in [k]$.

\medskip

Similar to the undirected case, relationships linking $\dic$, $\deltil$, and $\bic$ have recently gained interest. 
This started with a generalization of Brooks' theorem obtained first by Jacob and Meyniel~\cite{jacobCM75} and rediscovered independently by Mohar~\cite{moharLAA432}, which implies that every digraph $D$ satisfies
\[
    \dic(D) \leq \max \{ \lceil\deltil(D)\rceil, \bic(D)\},
\]
unless $\deltil(D) = 1$ and $D$ contains a directed cycle, or $\deltil(D) = 2$ and $D$ contains a symmetric odd cycle (see also~\cite{aboulkerDM113193} for alternative proofs and~\cite{bangSIDMA36,gonccalves2024brooks,aboulker2023digraph,picasarriArXiv25} for more general results). 

\medskip

Going one step further, Harutyunyan, Kawarabayashi, Picasarri-Arrieta, and Puig i Surroca~\cite{harutyunyanArXiv25} recently obtained that every digraph $D$ with $\deltil(D)$ being sufficiently large satisfies
\[
    \dic(D) \leq \max \{ \lceil\deltil(D)\rceil-1, \bic(D)\},
\]
unless $D$ is a very specific obstruction. This generalizes the aforementioned result of Reed~\cite{reedJCTB76}, which itself supports the conjecture of Borodin and Kostochka~\cite{borodinJCTB23}. Moreover, Kawarabayashi and Picasarri-Arrieta~\cite{kawarabayashiSODA25} recently posed the following conjecture, which, if true, implies both Conjecture~\ref{conj:reed} and an independent conjecture posed by Harutyunyan and Mohar~\cite{harutyunyanEJC18}.

\begin{conjecture}[\cite{kawarabayashiSODA25}]
    \label{conj:kpa}
    Every digraph $D$ satisfies $\dic(D) \leq \lceil \tfrac{1}{2}(\deltil(D) + 1 + \bic(D))\rceil$.
\end{conjecture}

Generalizing Corollary~\ref{cor:reed_1}, in the same paper they obtained the following, whose proof is inspired by the recent short proof of Corollary~\ref{cor:reed_1} due to King and Reed~\cite{kingJGT81}.

\begin{theorem}[\cite{kawarabayashiSODA25}]
    \label{thm:kpa}
    There exists $\epsilon>0$ such that every digraph $D$ satisfies 
    \[
    \dic(D) \leq \lceil (1-\epsilon)(\deltil(D) + 1) + \epsilon \bic(D)\rceil.
    \]
\end{theorem}

Building on a result of Picasarri-Arrieta~\cite{picasarriJGT106}, they also obtained that every oriented graph $D$  (that is, a digraph with biclique number~$1$) with underlying graph $G$ satisfies 
$\dic(D) \leq \tfrac{1}{3}\Delta(G)+2$ and $\dic(D) \leq \tfrac{1}{\sqrt{2}} \deltil(D) +2$ (see~\cite{kawarabayashiArxiv24}).
This improves on earlier bounds obtained respectively by Harutyunyan and Mohar~\cite{harutyunyanEJC18}, Golowich~\cite{golowichDM339}, and Steiner~\cite{steinerDAM287}.

\medskip

It follows from Theorem~\ref{thm:kpa} that every digraph $D$ with 
$\bic(D) \leq \deltil(D)-\frac{1}{\epsilon}(b+2)$
satisfies $\dic(D) \leq \deltil(D) - b$. In view of Corollary~\ref{cor:reed_2} and Conjectures~\ref{conj:reed_bis} and~\ref{conj:kpa}, this condition is certainly not optimal. In the present paper, we prove the following sufficient condition for a digraph $D$ to be $(\deltil(D)-b)$-colorable, which generalizes Corollary~\ref{cor:reed_2} to the directed setting and supports Conjecture~\ref{conj:kpa}.

\begin{restatable}{theorem}{mainthm}
    \label{thm:main}
    For every $b\in \mathbb{Z}$, there exists $\Delta_b\in \mathbb{N}$ such that every digraph $D$ with $\deltil(D) \geq \Delta_b$ and $\bic(D) \leq \deltil(D)-2b$ satisfies $\dic(D) \leq \deltil(D)-b$.
\end{restatable}

Corollary~\ref{cor:reed_2} indeed corresponds to the restriction of Theorem~\ref{thm:main} to symmetric digraphs. However, we note that Corollary~\ref{cor:reed_2} is known to hold with $\Delta_b$ being at most linear in $b$, while our proof requires at least $\Delta_b \geq 2^{\Omega(b^2)}$. We are aware that sharper arguments could likely reduce the required value of $\Delta_b$ substantially. However, we chose to prioritize the simplicity of the proof, particularly since a linear bound seems out of reach with our approach.
In particular, our argument is somewhat shorter than Reed’s original one.

We further note that, in the statement of Theorem~\ref{thm:main}, in contrast to the undirected case, the condition $\bic(D) \leq \deltil(D)-2b$ is best possible. Indeed, for all pairs of integers $b,\Delta\in \mathbb{N}$, there exists a digraph $D$ such that  $\deltil(D) \geq \Delta$, which satisfies $\bic(D) = \deltil(D) -2b+1$ and $\dic(D) = \deltil(D) -b+1$.
Such a digraph can be obtained from a biclique on  $\Delta+1$ vertices by blowing up $b$ vertices into directed triangles.

Finally, let us mention that we decided to state Theorem~\ref{thm:main} in terms of the parameter $\deltil$ because it is in line with numerous papers~\cite{harutyunyanEJC18,golowichDM339,kawarabayashiSODA25,harutyunyanArXiv25}, but that the same statement holds for many other degree parameters. To see this, for any function $f\colon \mathbb{N}^2 \to \mathbb{R}$, let $\Delta_f$ be the maximum degree parameter corresponding to $f$, that is, for every digraph $D$, let
\[
    \Delta_f(D) = \max_{v\in V(D)} f(d^-(v),d^+(v)),
\]
Our proof can easily be adapted to show that Theorem~\ref{thm:main} holds with $\Delta_f$ instead of $\deltil$ whenever $f$ is such that:
\begin{itemize}
    \item for every $x,y\in \mathbb{N}$, $f(x,y) \geq \min(x,y)$, and
    \item for every $a\in \mathbb{N}$, there exist $b,\Delta_0\in \mathbb{N}$ such that, for every integer $\Delta\geq \Delta_0$, we have $f(\Delta-a, \Delta+b) \geq \Delta$ and $f(\Delta+b, \Delta-a) \geq \Delta$.
\end{itemize}
For instance, Theorem~\ref{thm:main} holds when $\deltil$ is replaced by $\delmax$ or by the maximum arithmetic mean of the in- and out-degrees.

\medskip

One of the most natural degree parameters for which our result remains open is the maximum out-degree $\delplus$. Indeed, we do not even know whether it holds when restricted to oriented graphs. In this context, Kawarabayashi and Picasarri-Arrieta~\cite[Problem~7.1]{kawarabayashiSODA25} proposed the following open problem.

\begin{problem}[\cite{kawarabayashiSODA25}]
    Show the existence of $\epsilon>0$ and $\Delta_0\in \mathbb{N}$ such that every oriented graph $D$ with $\delplus(D)\geq \Delta_0$ satisfies $\dic(D) \leq (1-\epsilon)\delplus(D)$.
\end{problem}

It is a consequence of the Directed Brooks Theorem~\cite{jacobCM75} that every oriented graph $D$ with $\delplus(D)\geq 2$ satisfies $\dic(D)\leq \delplus(D)$.
It is also a consequence of a stronger result due to Harutyunyan, Kawarabayashi, Picasarri-Arrieta, and Puig i Surroca~\cite{harutyunyanArXiv25} (see~\cite[Theorem 12]{harutyunyanArXiv25}) that $\dic(D)\le \delplus(D)-1$ holds when $D$ is an oriented graph with $\delplus(D)$ being sufficiently large. 
However, we note that already the following problem is open.

\begin{problem}
    Show the existence of $\Delta_0\in \mathbb{N}$ such that every oriented graph $D$ with $\delplus(D)\geq \Delta_0$ satisfies $\dic(D) \leq \delplus(D)-2$.
\end{problem}

Bounding the dichromatic number of oriented graphs in terms of their maximum degree is at the heart of the following famous conjecture due to Erd\H{o}s and Neumann-Lara (see~\cite{erdosPNCN1979}), which can be seen as a directed version of Johansson's bound for triangle-free graphs~\cite{johansson1996}. 

\begin{conjecture}[\cite{erdosPNCN1979}]
    Every oriented graph $D$ satisfies $\dic(D) = O(\Delta/ \log \Delta)$, where $\Delta$ is the maximum degree of the underlying graph of $D$.
\end{conjecture}

The analogous conjecture in terms of $\deltil$ has been posed by McDiarmid and Mohar, see~\cite{harutyunyanEJC18}. 

\section{Technical Overview}

Our proof of Theorem~\ref{thm:main} combines structural and probabilistic arguments. 
We argue by contradiction and let $D$ be a smallest counterexample. 
Our starting point is a partition of the vertex set of $D$ into one sparse part (that is, a set of vertices whose out-neighborhood is sparse) and a collection of dense subdigraphs.
This is done using a recent Dense Decomposition Lemma due to Harutyunyan, Kawarabayashi, Picasarri-Arrieta, and Puig i Surroca~\cite{harutyunyanArXiv25}. 
Previous work~\cite{harutyunyanEJC18,kawarabayashiSODA25,harutyunyanArXiv25} already shows that the sparse vertices can be handled by a suitable random dicoloring argument. 
Consequently, the main difficulty lies in handling the dense sets.
Harutyunyan, Kawarabayashi, Picasarri-Arrieta, and Puig i Surroca~\cite{harutyunyanArXiv25} developed the structural machinery based on the notion of saviors, which they use to prove the case $b=2$ (and actually a stronger statement characterizing the exact obstructions to the ($\deltil-1$)-dicolorability). 
Without any further ingredients, this machinery naturally yields only the weaker statement that all digraphs with biclique number $\deltil - \Omega(b)$ have dichromatic number at most $\deltil-b$. 
The main novelty of the present paper is the introduction of vertex identifications, which provides the additional flexibility needed to extend this approach and obtain the exact bound for every fixed $b$.

We first show that every dense set contains a biclique on $(1-o(1))\deltil$ vertices, with at most $O(b)$ exceptional vertices.
The key notion in the structural part of the proof is that of a {\it savior}. Informally, a savior is a vertex inside the biclique that has sufficiently many neighbors outside it which are essentially ``independent'' of the biclique itself. 
Moreover, a savior needs to be in the common neighborhood of all exceptional vertices of the dense set.
The fact that the exceptional set has size $O(b)$ is crucial here. Since the Dense Decomposition Lemma is applied with a parameter $\epsilon_b$ chosen sufficiently small with respect to $b$, the common neighborhood of all exceptional vertices inside the biclique remains large. This allows us to choose many saviors that are adjacent to all exceptional vertices.
 
In the random partial dicoloring constructed later, a savior is likely to see many colors appearing both inside and outside the biclique.
Typically, a savior $v$ with out-degree $\deltil$ has at least $b+1$ colors repeated in its out-neighborhood, making it easy to color.
More precisely, no matter how we extend the coloring to the rest of its out-neighborhood, there will always be at least one color which remains free for $v$.
Using the minimality of our counterexample, we obtain that every dense set contains many saviors, roughly at least half of its vertices.
This is large enough so that, with high probability, in the random partial dicoloring, many saviors will be {\it rescuers}, that is, saviors that are actually easy to color.

To get a bit more into details, we need to distinguish two kinds of dense sets, namely the {\it loose} and {\it tight} ones.
A dense set is tight if its biclique is large, namely if it contains at least $\deltil -3b+2$ vertices, while the biclique of a loose set can contain from $\deltil - O_b(\log^3(\deltil))$ to $\deltil-3b+1$ vertices. The distinction between loose and tight dense sets lies not in the existence of saviors, but in how much flexibility they provide.

Since a loose dense set contains relatively few vertices, each of its saviors has many neighbors outside the dense set, and this external flexibility alone is sufficient to complete the random dicoloring.
The tight sets are considerably more delicate, and their structural analysis is the main technical contribution of the paper. Here, the large biclique has size $\deltil-O(b)$, so although there are still many saviors, they may have only a few neighbors outside the dense set, which is no longer sufficient to complete the coloring directly. 
To recover the missing flexibility, we force carefully chosen pairs of exceptional vertices to receive the same color. Since every savior is adjacent to all exceptional vertices, identifying two exceptional vertices creates a repeated color in the neighborhood of every savior. These additional repeated colors compensate for the lack of neighbors outside the dense set.

The probabilistic part is thus carried out by coloring randomly the vertices of an auxiliary digraph $D^\star$, obtained by identifying the aforementioned pairs of vertices inside the dense sets.
Formally, we give a color to each vertex of $D^\star$ uniformly at random, and then we uncolor all vertices of $D^\star$ whose color appears in both their in- and out-neighborhoods. This yields a random partial dicoloring of $D^\star$, and we interpret it as a random partial dicoloring of the original digraph, thereby introducing exactly the color equalities encoded by the identifications.

The crucial property is that tight dense sets admit enough additional structure compared to the loose sets, so that $D^\star$ does not differ too much from $D$. In particular, the required identifications increase the maximum degree by only $O(b)$ and preserve the sparsity of the sparse vertices. 
The structural properties established in the previous sections guarantee that every bad event, which is of the form ``the random partial dicoloring cannot be extended to the $i$th dense set'' or ``the random partial dicoloring cannot be extended to the $i$th sparse vertex'' is sufficiently unlikely. 
Moreover, each bad event depends on only polynomially many other bad events, while its probability is super-polynomially small, which allows us to apply the Lovász Local Lemma.
The polynomial bound on the dependencies of the bad events ultimately relies on the fact that the auxiliary digraph $D^\star$ preserves distances up to a factor of three, which in turn follows from the refined structure of tight dense sets.

\section{Preliminaries}

\subsection{Definitions and notation}

For any positive integer $k$, we denote by $[k]$ the set of integers $\{1,\dots,k\}$. For any set $X$, we denote by $\binom{X}{2}$ the set of pairs of elements of $X$.

Our notation on digraphs follows~\cite{bang2009}. 
Let $D$ be a digraph. The vertex-set of $D$ is denoted by $V(D)$ and its arc-set by $A(D)$.
A \textit{digon} of $D$ is a pair of arcs of $D$ in opposite directions between the same vertices. A digraph is {\it symmetric} if each of its arcs belongs to a digon.
The \textit{underlying graph} of $D$ is the undirected graph with vertex-set $V(D)$ in which $uv$ is an edge if and only if $uv$ or $vu$ is an arc of $D$.
The {\it complement} of $D$, denoted by $\overline{D}$, is the digraph with vertex set $V(D)$ that contains all arcs except those of $D$.
A {\it matching} of $D$ is a set of pairwise disjoint edges of the underlying graph $G$ of $D$. We denote by $\nu(D)$ the maximum cardinality of a matching of $D$.
 
Let $v$ be a vertex of $D$. The {\it out-neighborhood} of $v$ in $D$, denoted by $N^+_D(v)$, is the set of vertices $w\in V(D)$ such that $vw\in A(D)$. 
Similarly, the {\it in-neighborhood} $N^-_D(v)$ of $v$ is the set of vertices $u$ such that $uv\in A(D)$. 
The {\it out-degree} $d^+_D(v)$ and the {\it in-degree} $d^-_D(v)$ of $v$ are the number of out-neighbors and in-neighbors of $v$, respectively.
We denote by $N_D(v)$ the {\it neighborhood} of $v$ in $D$, which is the union of its in- and out-neighborhoods. 
We finally denote by $N_D^\pm(v)$ the set of vertices that are linked with $v$ by a digon, that is, $N_D^\pm(v) = N^+_D(v) \cap N^-_D(v)$. 
For each of the notations above, we omit the subscript when $D$ is clear from the context.
We use two notions of maximum degrees for $D$, namely:
\[
    \delmax(D) = \max_{v\in V(D)} \max  \{ d^+(v), d^-(v)\} \hspace{0.8cm} \text{and} \hspace{0.8cm} \deltil(D) = \max_{v\in V(D)} \sqrt{d^-(v)\cdot d^+(v)}.
\]

Let $X\subseteq V(D)$ be any subset of vertices of $D$. The subdigraph of $D$ {\it induced} by $X$, denoted by $D[X]$, is the digraph with vertex-set $X$ containing all arcs of $D$ with both extremities in $X$. 
We denote by $D/X$ the digraph obtained from $D$ by {\it identifying} $X$ into a single vertex $s$. Formally, $D/X$ has vertex-set $(V(D) \setminus X) \cup \{s\}$ and contains all arcs of $D[V(D)\setminus X]$ as well as the arcs
\[
    \{us : ux\in A(D) \text{ for some $x\in X$}\} \cup \{s u : xu\in A(D) \text{ for some $x\in X$}\}.
\]

Let $Y\subseteq V(D)$ be disjoint from $X$. We say that $X$ {\it dominates} $Y$ and that $Y$ is {\it dominated} by $X$ if $D$ contains all arcs of the form $xy$ for $x\in X$ and $y\in Y$. We further say that a vertex $x$ dominates $Y$ if $\{x\}$ dominates $Y$, and similarly that a vertex $y$ is dominated by $X$ if $\{y\}$ is dominated by $X$.

A \textit{biclique} of $D$ is a set of vertices inducing a complete digraph, which is a digraph containing all possible arcs.
The \textit{biclique number} $\bic(D)$ of $D$ is the size of the largest biclique of $D$.
For any $\ell$, let $\cSCR_\ell$ be the digraph with vertex-set $\{u_0,u_1,\dots,u_{\ell-1}\}$ and arc-set $\{u_iu_{i+1 \bmod \ell}: 0\leq i \leq \ell-1\}$.
A {\it directed cycle} is a digraph isomorphic to $\cSCR_\ell$ for some $\ell\geq 2$.
A digraph is {\it acyclic} if it does not contain any directed cycle.
Given an integer $k$, a \textit{$k$-dicoloring} of $D$ is a coloring of its vertex-set $\phi\colon V(D) \to [k]$ such that each color class $\phi^{-1}(i)$ induces an acyclic subdigraph on $D$. The \textit{dichromatic number} of $D$, denoted by $\dic(D)$, is the smallest $k\in \mathbb{N}$ for which $D$ admits a $k$-dicoloring.

\subsection{Dicoloring tools}
\label{subsec:dic_tools}

Along the proof, we often make use of the following easy and well-known observations. 

\begin{lemma}
    \label{lemma:coloring_contraction}
    Let $D$ be a digraph, $S\subseteq V(D)$ be such that $D[S]$ is acyclic, $D/S$ be the digraph obtained by identifying $S$ into a single vertex $s$, and $\phi^\star$ be a dicoloring of $D/S$. Then 
    \[
    \phi\colon v\mapsto 
    \begin{cases}
        \phi^\star(v) & \text{if $v\notin S$}\\
        \phi^\star(s) & \text{otherwise}
    \end{cases}
    \]
    is a dicoloring of $D$.
\end{lemma}
\begin{proof}
    Assume for a contradiction that $D$, colored with $\phi$, contains a monochromatic directed cycle $\cSCR$. Since $S$ is acyclic, $V(\cSCR) \nsubseteq S$. Hence, either $\cSCR$ is disjoint from $S$, or it contains a directed path that is internally disjoint from $S$, with both extremities in $S$.  In both cases, $D/S$ colored with $\phi^\star$ contains a monochromatic directed cycle, a contradiction.
\end{proof}

\begin{lemma}
    \label{lemma:dicoloring_DAG}
    Let $D$ be a digraph, $S\subseteq V(D)$ be such that $D[S]$ is acyclic, and $\phi^\star$ be a dicoloring of $D-S$. If $\alpha$ is a color that is not used by $\phi^\star$ in $\bigcup_{s\in S}N^+(s)$ or in $\bigcup_{s\in S}N^-(s)$, then 
    \[
    \phi\colon v\mapsto 
    \begin{cases}
        \phi^\star(v) & \text{if $v\notin S$}\\
        \alpha & \text{otherwise}
    \end{cases}
    \]
    is a dicoloring of $D$.
\end{lemma}
\begin{proof}
    Assume for a contradiction that $D$, colored with $\phi$, contains a monochromatic directed cycle $\cSCR$. Then $V(\cSCR) \nsubseteq S$, as $D[S]$ is acyclic. Furthermore, $V(\cSCR) \nsubseteq V(D)\setminus S$, for otherwise $\cSCR$ would be a monochromatic directed cycle of $D-S$ colored with $\phi^\star$. 
    Therefore, $\cSCR$ contains an arc from $S$ to $V(D)\setminus S$ and an arc from $V(D)\setminus S$ to $S$. Since $\cSCR$ is monochromatic, this shows that $\alpha$ appears in both $\bigcup_{s\in S}N^+(s)$ and $\bigcup_{s\in S}N^-(s)$, a contradiction. 
\end{proof}

In particular, Lemma~\ref{lemma:dicoloring_DAG} implies that one can greedily color a digraph by repeatedly choosing an uncolored vertex $v$, and assigning it a color that does not appear in $N^-(v)$ or in $N^+(v)$. This easy procedure shows that every digraph $D$ satisfies $\dic(D) \leq \deltil(D)+1$.

Another useful application of Lemma~\ref{lemma:dicoloring_DAG} is the following. Suppose $\{u,v\}$ is a pair of uncolored vertices with at most one arc between them. One may then assign both vertices the same color, provided that this color does not appear in $N^+(u) \cup N^+(v)$ or in $N^-(u) \cup N^-(v)$. This is useful when $u$ and $v$ have many common out-neighbors or in-neighbors, since in that case: (i) a common color is likely to exist, and (ii) every uncolored common out-neighbor acquires a repeated color in its in-neighborhood, making it easier to color afterwards using the greedy procedure above.

Throughout the proof, we use these two procedures repeatedly without explicitly referring to Lemma~\ref{lemma:dicoloring_DAG}, for the sake of conciseness.

\subsection{Dense decompositions}

For the structural part of our proof, we use the following lemma due to Harutyunyan, Kawarabayashi, Picasarri-Arrieta, and Puig i Surroca~\cite{harutyunyanArXiv25}, for which we first need a specific definition.
Suppose that $D$ is a digraph with $\delmax(D) = \delm$, and let $d$ be any real number. We say that a vertex $v$ is {\it $d$-sparse} (with respect to $D$) if the digraph induced by its out-neighborhood contains at most $\delm(\delm-1) - d\delm$ arcs. 
The following lemma appears as a particular case of~\cite[Lemma~9]{harutyunyanArXiv25}. It can be seen as a general directed form of the so-called dense decompositions of graphs appearing in Molloy and Reed's series of papers, see for instance~\cite{reedJGT27, molloyCPC7,molloyComb18, reedJCTB76, farzadJCTB93, molloyJCTB109}. 
Informally, it says that every digraph with sufficiently large maximum degree $\delm$ can be decomposed into one set of sparse vertices and a collection of pairwise disjoint ``dense sets'' that behave like bicliques on roughly $\delm$ vertices.

\begin{lemma}[{\sc Dense Decomposition Lemma}~\cite{harutyunyanArXiv25}]
    \label{lemma:dense_decomposition}
    For every $0 < \epsilon < \frac{1}{2}$ there exists $\DeltaDecomp(\epsilon)$ such that the following holds.
    Let $D$ be a digraph with $\delmax(D) = \delm \geq \DeltaDecomp(\epsilon)$ and let $\dm = \log^3(\delm)$. Then $D$ admits a partition $(X_1, \ldots, X_t, S)$ of its vertex-set such that:
    \begin{enumerateAlpha}
        \item for every $i\in [t]$, $\delm- \frac{3}{\epsilon} \dm < |X_i| < \delm+1 + 4\dm$;
        \label{enum:dense_decomposition:1}
        \item for every $i\in [t]$ and $u\in V(D)$,
        $u \in X_i$ if and only if $|N^+(u) \cap X_i| \geq (1-\epsilon) \delm$; and
        \label{enum:dense_decomposition:3}
        \item vertices in $S$ are $\dm$-sparse.
        \label{enum:dense_decomposition:4}
    \end{enumerateAlpha}
\end{lemma}

From now on, for every $\epsilon \in  (0,\tfrac{1}{2})$, we denote by $\DeltaDecomp(\epsilon)$ the smallest integer for which Lemma~\ref{lemma:dense_decomposition} holds.
Intuitively speaking, the lemma above is useful when one wishes to extend some result known for sparse digraphs (that is, for digraphs in which all vertices are sparse) to a more general class of digraphs, as it provides structure on the vertices that are not sparse. This is typically our case with Theorem~\ref{thm:main}: if $D$ is a $d$-sparse digraph (that is, all vertices in $D$ are $d$-sparse) for some $d>\log^3(\Delta)$, then it is known that $D$ has dichromatic number at most $\Delta-\Omega(d)$, see~\cite[Theorem~1.4]{kawarabayashiSODA25} (with $\Delta$ being $\delmax(D)$ here).

\subsection{Probabilistic tools}

As mentioned above, our proof contains probabilistic arguments. The reader unfamiliar with the probabilistic method is referred to~\cite{alon2016book,molloyreed}.
We make use of the following three well-known results.
We first need the symmetric version of the Lov\'asz Local Lemma, due to Erd\H{o}s and Lov\'asz~\cite{erdosIFS10}.

\begin{lemma}[{\sc Lov\'asz Local Lemma}~\cite{erdosIFS10}]
\label{lemma:lll}
    Let $A_1,A_2,\dots,A_n$ be events in an arbitrary probability space. Suppose that each event $A_i$ is mutually independent of a set of all the other events but at most $d$, and that $\PP(A_i) \leq p$ for all $1\leq i \leq n$. If 
    $4pd \leq 1$ then $\PP\left(\bigcap_{i=1}^n \overline{A_i}\right)>0$.
\end{lemma}

We further make use of the following consequence of the celebrated concentration inequality of Talagrand~\cite{talagrand1995}. A proof of the following statement can be found in~\cite[Lemma~17]{harutyunyanArXiv25}.

\begin{lemma}[{\sc Talagrand}~\cite{talagrand1995}]
    \label{lemma:talagrand}
    Let $X$ be a random variable valued in $\mathbb{N}$, determined by $n$ independent trials and satisfying the following for some integer $r \geq 1$:
    \begin{enumerateNum}
        \item changing the outcome of any one trial can affect $X$ by at most $1$, and
        \item for every $k\in \mathbb{N}$, if $X\geq k$ then there is a set of at most $rk$ trials whose outcomes certify that $X\geq k$.
    \end{enumerateNum}
    Then, for any real $\lambda>126\sqrt{r\EE(X)} + 344 r$, we have $\PP\left(|X- \EE(X)| > \lambda \right) \leq 4\exp\left(\frac{-\lambda^2}{32 r(\EE(X) + \lambda)} \right)$.
\end{lemma}

We finally make use of the celebrated concentration inequality of Azuma~\cite{azumaTMJ19}. The following statement can be found in~\cite[Chapter~11]{molloyreed}.

\begin{lemma}[{\sc Azuma}~\cite{azumaTMJ19}]
    \label{lemma:azuma}
    Let $X$ be a random variable determined by $n$ trials $T_1,\ldots,T_n$, such that, for each $1\leq i\leq n$, and any two possible sequences of outcomes $t_1,\ldots,t_i$ and $t_1,\ldots,t_{i-1},t'_i$,
    \[\Big|\,\EE(X\,|\, T_1=t_1,\ldots,T_i=t_i)-\EE(X\,|\, T_1=t_1,\ldots,T_{i-1}=t_{i-1},T_i=t'_i)\,\Big|\leq \delta_i.\]
    Then, for any real $\lambda\geq 0$, we have $\PP(|X-\EE(X)|>\lambda)\leq 2\exp\left(\frac{-\lambda^2}{2\sum_{i=1}^n \delta_i^2}\right)$.
\end{lemma}

\section{Proof of Theorem~\ref{thm:main}}

The remainder of the paper is devoted to the proof of Theorem~\ref{thm:main}. 
Roughly speaking, the proof proceeds in three steps, organized as follows.
We first take a minimum counterexample $D$, and exhibit a collection of structural properties on $D$. These structural properties arise from the combination of the Dense Decomposition Lemma (applied to $D$ with some $\epsilon$ small enough for a fixed $b$) and the minimality of $D$. 
Once enough structure is obtained on $D$, we manage to build an auxiliary graph $D^\star$, whose structure is smoother than that of $D$, with the guarantee that any dicoloring of $D^\star$ with few colors yields a dicoloring of $D$ with the same number of colors.
We then use a pseudo-random coloring process to conclude on the colorability of $D^\star$, and thus the one of $D$, hence reaching a contradiction. 

For convenience, let us first restate Theorem~\ref{thm:main} in the following equivalent form.

\begin{theorem}
	\label{thm:main_reform}
    For every $b\in \mathbb{Z}$, there exists $\Delta_b\in \mathbb{N}$ such that every digraph $D$ with $\deltil(D) \geq \Delta_b$ and $\bic(D) \leq \deltil(D)-2b+2$ satisfies $\dic(D) \leq \deltil(D)-b+1$.
\end{theorem}

\begin{proof}
    Let us fix $b\in \mathbb{Z}$. As explained in Section~\ref{subsec:dic_tools}, it is well-known that every digraph $D$ satisfies $\dic(D)\leq \deltil(D)+1$. We henceforth assume that $b\geq 1$. 
    From now on, let $\epsilonb = \frac{1}{24b}$. We let $\Delta_b$ be the smallest integer so that, for every real number $\Delta\geq \Delta_b$, each of the following inequalities holds:
    \begin{subequations}
    \renewcommand{\theequation}{$\Delta$\arabic{equation}}
    \begin{align}
        \Delta &\geq \DeltaDecomp(\epsilonb) + b-1,\label{eq:largeDelta:1}\\
        2\log^3(\Delta-b+1) &\geq \log^3(\Delta) \geq \tfrac{4}{5}\log^3(\Delta+b+1) +b+1,\label{eq:largeDelta:2}\\
        \epsilonb^2\Delta &> \epsilonb\log^4(\Delta) > 16\log^3(\Delta) + 240b^2,\label{eq:largeDelta:3}\\
        \log(\Delta) &\geq 2^8b,\label{eq:largeDelta:4}\\
        \Delta\log^3(\Delta) &\geq 378b \Delta + 2^{35}\cdot 715^2 \Delta +  3b\log^3(\Delta) + 108b^2,\label{eq:largeDelta:5}\\
        \left(1-\tfrac{6b}{\Delta}\right)^{\Delta/6b} &\geq 1/4, \label{eq:largeDelta:6}\\
        \exp(-\log^2(\Delta))& \geq \max(8\exp(-2^{-50}\log^3(\Delta)),2\exp(-\Delta^{1/3}/144)),\label{eq:largeDelta:8}\\
        \Delta^{1/3} & \geq b\cdot 2^{117b^2}\cdot \log^4(\Delta),\label{eq:largeDelta:9}\\
       \exp(\log^2(\Delta)) &\geq 2^{19} \cdot (\Delta+b+1)^{13}\label{eq:largeDelta:10}.
    \end{align}
    \end{subequations}
    Note that, for~\eqref{eq:largeDelta:6}, we use that $\lim_{x\to +\infty} (1-\tfrac{1}{x})^x = 1/\e > 1/4$.
    We claim that the statement holds for this particular value of $\Delta_b$. Assume that this is not the case, so there exists an integer $\Delta \geq \Delta_b$ and a digraph $H$ with:
    \begin{itemize}
        \item $\lfloor\deltil(H)\rfloor \leq \Delta$,
        \item $\bic(H) \leq \Delta-2b+2$, and
        \item $\dic(H) \geq \Delta-b+2$.
    \end{itemize}
    Among all such digraphs $H$, we choose $D=(V,A)$ of minimum order.
    In particular, it follows that every proper induced subdigraph of $D$ has dichromatic number at most $\Delta-b+1$.
    For the sake of conciseness, from now on, we let $\delm$ denote $\delmax(D)$ and $d= \log^3(\Delta)$.

    \subsection{Application of the Dense Decomposition Lemma}

    In this first part, we exhibit some structural properties of $D$, arising from the combination of the dense decomposition lemma and the fact that $D$ is a minimum counterexample to the desired statement. We first note that $D$ being a minimum counterexample implies the following degree conditions.

    \begin{claim}
        \label{claim:degrees}
        Let $v$ be an arbitrary vertex, then
        \begin{enumerate}
            \item $\Delta-b+1 \leq \min \{d^+(v),d^-(v) \} \leq \Delta$, and
            \label{enum:degrees_1}
            \item $\max \{d^+(v), d^-(v)\} \leq \Delta+b+1$.
            \label{enum:degrees_2}
        \end{enumerate}
    \end{claim}
    \begin{proofclaim}
        Let $v$ be an arbitrary vertex.
        Assume first that $\min(d^-(v), d^+(v)) \leq  \Delta - b$. By choice of $D$, there exists a $(\Delta-b+1)$-dicoloring $\phi^\star$ of $D-v$, which can be extended to $D$ by choosing for $v$ a color of $[\Delta-b+1]$ that is not appearing in its in- or out-neighborhood. This contradicts $\dic(D) > \Delta-b+1$, showing the first inequality.

		Next, if $\min \{d^+(v),d^-(v)\} \geq \Delta+1$ then $\sqrt{d^+(v) \cdot d^-(v)} \geq \Delta+1$, which in particular implies that $\lfloor\deltil(D)\rfloor \geq \Delta+1$, a contradiction. This shows~\ref{enum:degrees_1}.
        
        For~\ref{enum:degrees_2}, assume for a contradiction that $\max \{d^+(v),d^-(v)\} \geq \Delta+b+2$, then
        \[
            d^+(v)\cdot d^-(v) \geq (\Delta-b+1) \cdot (\Delta+b+2) = (\Delta+1)^2 +(\Delta +1)-b(b+1) \geq (\Delta+1)^2,
        \]
        and in particular $\lfloor\deltil(D)\rfloor \geq \Delta+1$, a contradiction. The claim follows.
    \end{proofclaim}

    In particular, Claim~\ref{claim:degrees} shows that $\delm$ is close to $\Delta$, which allows us to apply the Dense Decomposition Lemma and derive the following.

    \begin{claim}
        \label{claim:dense_decomposition}
        There exists a partition $(X_1,\dots,X_t,S)$ of $V$ such that:
        \begin{enumerate}
            \item for every $i\in [t]$, $ \Delta - \frac{4}{\epsilonb}d < |X_i| < \Delta + 5d$;
            \label{claim:dense_decomposition:i}
            \item for every $i\in [t]$ and $u\in X_i$, $|N^+(u) \cap X_i| \geq (1-\epsilonb)\Delta -b$;
            \label{claim:dense_decomposition:ii}
            \item for every $i\in [t]$ and $u\in V \setminus X_i$, $|N^+(u) \cap X_i| \leq (1-\epsilonb)\Delta +b+1$; and
            \label{claim:dense_decomposition:iii}
            \item vertices in $S$ are $(d/2)$-sparse.
            \label{claim:dense_decomposition:iv}
        \end{enumerate}
    \end{claim}
    \begin{proofclaim}
        By Claim~\ref{claim:degrees}, we have that $\Delta-b+1 \leq \delm \leq \Delta+b+1$. 
        Therefore, Lemma~\ref{lemma:dense_decomposition} can be applied with $\epsilon = \epsilonb$, as 
        \[
        \delm \geq \Delta-b+1 \geq \DeltaDecomp(\epsilonb)
        \]
        by~\eqref{eq:largeDelta:1}.
        Hence, by Lemma~\ref{lemma:dense_decomposition}, there exists a partition $(X_1,\dots,X_t,S)$ of $V$ such that:
        \begin{enumerateAlpha}
            \item for every $i\in [t]$, $\delm - \frac{3}{\epsilonb} \log^3(\delm) < |X_i| < \delm + 1+ 4\log^3(\delm)$;
            \label{proofclaim:dense_decomposition:a}
            \item for every $i\in [t]$ and $u\in V$, $u\in X_i$ if and only if $|N^+(u) \cap X_i| \geq (1-\epsilonb)\delm$;
            \label{proofclaim:dense_decomposition:b}
            \item vertices in $S$ are $\log^3(\delm)$-sparse.
            \label{proofclaim:dense_decomposition:c}
        \end{enumerateAlpha}
        
        We claim that the same partition satisfies the statement of the claim. Recall that $\Delta-b+1 \leq \delm \leq \Delta+b+1$. 
        Therefore, by~\ref{proofclaim:dense_decomposition:a}, for every $i\in [t]$, we have
        \[
            |X_i| > \Delta-b - \tfrac{3}{\epsilonb}\log^3(\Delta+b+1) \geq  \Delta - \tfrac{4}{\epsilonb}\log^3(\Delta)
        \]
        by~\eqref{eq:largeDelta:2}.
        Similarly,  for every $i\in [t]$, we have
        \[
            |X_i| < \Delta + b +1 + 4\log^3(\Delta+b+1) \leq \Delta + 5\log^3(\Delta),
        \]
        which shows~\ref{claim:dense_decomposition:i}.
        Let $i\in [t]$ be an arbitrary index and $u$ be an arbitrary vertex in $X_i$. Then by~\ref{proofclaim:dense_decomposition:b} we have
        \[
            |N^+(u) \cap X_i| \geq (1-\epsilonb)\delm \geq (1-\epsilonb)(\Delta-b+1)\geq (1-\epsilonb)\Delta-b,
        \]
        which shows~\ref{claim:dense_decomposition:ii}.
        Similarly, by~\ref{proofclaim:dense_decomposition:b}, for any vertex $u\in V\setminus X_i$ we have 
        \[
            |N^+(u) \cap X_i| \leq (1-\epsilonb)\delm \leq (1-\epsilonb)(\Delta+b+1)\leq (1-\epsilonb)\Delta+b+1,
        \]
        which shows~\ref{claim:dense_decomposition:iii}. Finally, \ref{claim:dense_decomposition:iv} follows from~\ref{proofclaim:dense_decomposition:c} and the fact that
        \[
            \log^3(\delm) \geq \log^3(\Delta-b+1) \geq \tfrac{1}{2}\log^3(\Delta)
        \]
        by~\eqref{eq:largeDelta:2}. The claim follows.
    \end{proofclaim}
    
    From now on, let us fix a partition $(X_1,\dots,X_t,S)$ of $V$ guaranteed by Claim~\ref{claim:dense_decomposition}. For every $i\in [t]$, we let $D_i = D[X_i]$ and $\nu_i = \nu(\overline{D_i)}$. Furthermore, from now on, for every $i\in [t]$, we fix a maximum matching $M_i=(u_jv_j)_{j\in [\nu_i]}$ of $\overline{D_i}$, and we let $U_i$ be the set of vertices spanned by $M_i$, that is
    \[
        U_i = \bigcup_{j\in [\nu_i]} \{u_j,v_j\}.
    \]
    We finally let $K_i = X_i \setminus U_i$. Observe that, by maximality of $M_i$, $K_i$ is a biclique.
    Our first goal is to bound $\nu_i$ for every $i\in [t]$, hence showing that each $D_i$ is almost a complete digraph. For this, we first need the following technical observation that we use several times later on.

    \begin{claim}
        \label{claim:extending_matching}
        Let $i\in [t]$, $M=(x_jy_j)_{j\in [\nu]}$ be a matching of $\overline{D_i}$ of size $\nu \leq 2b+1$, and $U$ be the set of vertices spanned by $M$. There exists a $(\Delta-b+1)$-dicoloring $\phi$ of $D-(X_i\setminus U)$ such that, for every $j\in [\nu]$, $\phi(x_j) = \phi(y_j)$.
    \end{claim}
    \begin{proofclaim}
        By minimality of $D$, there exists a $(\Delta-b+1)$-dicoloring $\phi$ of $D-X_i$. 
        By definition of a matching in $\overline{D_i}$,  for every $j$, there is at most one arc between $x_j$ and $y_j$ in $D$. We can thus extend $\phi$ to $U$ by assigning, for every $j\in [\nu]$, a common color to $\{x_j,y_j\}$ that is not already appearing in $N^+(x_j) \cup N^+(y_j)$. Note that such a color is available because $x_j$ and $y_j$ have a large fraction of their out-neighbors in $X_i$, which are still uncolored at that stage. Indeed, the number of already colored vertices in $N^+(x_j) \cup N^+(y_j)$ is at most
        \begin{align*}
            &|N^+(x_j)\setminus X_i| + |N^+(y_j) \setminus X_i| + |U|\\
            &= d^+(x_j) + d^+(y_j) - |N^+(x_j)\cap X_i| - |N^+(y_j) \cap X_i| + |U|\\
            &\leq 2(\Delta+b+1) -2((1-\epsilonb)\Delta - b) + (4b+2) &\text{by Claims~\ref{claim:degrees} and~\ref{claim:dense_decomposition}\ref{claim:dense_decomposition:ii},}\\
            &= 2\epsilonb\Delta + 8b+4\\
            &\leq \Delta-b &\text{by choice of $\epsilonb$ and~\eqref{eq:largeDelta:3}.}
        \end{align*}
        The claim follows.
    \end{proofclaim}

    We are now ready to bound $\nu_i$ for every $i\in [t]$. 
    
    \begin{claim}
        \label{claim:matching_Di}
        For every $i\in [t]$, $\nu_i \leq 2b$. In particular, $|U_i| \leq 4b$.
    \end{claim}
    \begin{proofclaim}
        Assume for a contradiction that $\overline{D_i}$ contains a matching of size $\nu=2b+1$, and let $M=(x_jy_j)_{j\in [\nu]}$ be such a matching. Note that we take $M$ of size precisely $2b+1$, even though there might exist larger matchings. Let $U$ be the set of vertices spanned by $M$.

        By Claim~\ref{claim:extending_matching}, there exists a $(\Delta-b+1)$-dicoloring $\phi$ of $D-(X_i\setminus U)$ such that $\phi(x_j) = \phi(y_j)$ for every $j\in [\nu]$. We show that $\phi$ can be extended to $D$, hence contradicting $\dic(D) > \Delta-b+1$. Let $L$ be the set of vertices in $X_i$ that are dominated by $U$, that is 
        \[
            L = X_i \cap \bigcap_{u\in U} N^+(u).
        \]
        In particular, note that $L$ is disjoint from $U$.
        Finally, let $R$ be the remaining vertices of $X_i$, so $R=X_i \setminus (U\cup L)$. Observe that $(U,L,R)$ partitions $X_i$. Let us first point out that $L$ is significantly large, as
        \begin{align*}
            |L| &\geq |X_i| - \sum_{u\in U}|X_i \setminus N^+(u)|\\&\geq |X_i| - |U| \cdot \big( |X_i| - ((1-\epsilonb)\Delta - b) \big) &\text{by Claim~\ref{claim:dense_decomposition}\ref{claim:dense_decomposition:ii},}\\
            &\geq \Delta -\tfrac{4}{\epsilonb}d - (4b+2)\cdot (\epsilonb \Delta + 5d +  b) &\text{by Claim~\ref{claim:dense_decomposition}\ref{claim:dense_decomposition:i},}\\
            &\geq \tfrac{1}{2}\Delta - (8b+4)\epsilonb \Delta &\text{by~\eqref{eq:largeDelta:3},}\\
            &\geq 2\epsilonb \Delta. &\text{by choice of $\epsilonb$.}
        \end{align*}
        This allows us to extend $\phi$ to $R$, by choosing, for each vertex $r\in R$, a color that is not already appearing in $N^+(r)$. 
        Note that this is possible, since the number of colored vertices in $N^+(r)$ is at most
        \begin{align*}
            |N^+(r) \setminus L| &= d^+(r) - |L \cap N^+(r)|\\
            &\leq (\Delta +b+1) - |L| + |X_i \setminus N^+(r)| &\text{by Claim~\ref{claim:degrees},}\\
            &\leq (\Delta +b+1) - 2\epsilonb \Delta + |X_i| - ((1-\epsilonb)\Delta-b) &\text{by Claim~\ref{claim:dense_decomposition}\ref{claim:dense_decomposition:ii},}\\
            &\leq (\Delta +b+1) -2\epsilonb \Delta + \epsilonb\Delta + 5d+b&\text{by Claim~\ref{claim:dense_decomposition}\ref{claim:dense_decomposition:i},}\\
            &= (1-\epsilonb)\Delta +5d+2b+1 \\
            &\leq \Delta-b &\text{by~\eqref{eq:largeDelta:3}.}
        \end{align*}
        We finally extend the partial dicoloring to $L$, by choosing for every vertex $x\in L$ a color that is not already appearing in $N^-(x)$. Recall that, by Claim~\ref{claim:degrees}, $d^-(x) \leq \Delta+b+1$. Moreover, by definition, $U$ dominates $x$, and by choice of $\phi$ at most $\frac{1}{2}|U|$ colors are used on $U$. Therefore, the number of colors appearing in the in-neighborhood of $x$ is at most
        \[
            d^-(x) - \frac{1}{2}|U| \leq \Delta +b+1 - |M| = \Delta -b.
        \]
        This shows that $\phi$ can be extended to $D$, a contradiction. Therefore, $\nu_i = |M_i| \leq 2b$, and by definition we have $|U_i| = 2|M_i| \leq 4b$, as desired.
    \end{proofclaim}

    We conclude this first part with the following technical claim, which is a consequence of Claim~\ref{claim:dense_decomposition}\ref{claim:dense_decomposition:iii} that we use several times later on. 

    \begin{claim}
        \label{claim:matching_S}
        For every $i\in [t]$, $X\subseteq X_i$, and $L \subseteq V(D) \setminus X_i$, if $|X|\geq (1-\tfrac{1}{2}\epsilonb)\Delta$ and $|L|\leq 2b^2+b$, then there exists a matching of size $|L|$ between $X$ and $L$ in $\overline{D}$.
    \end{claim}
    \begin{proofclaim}
        Let $u_1,\dots,u_{\ell}$ be an arbitrary labeling of $L$, where $\ell=|L|$, and let us show that there exists, in $\overline{D}$, a matching $x_1u_1,\dots,x_\ell u_\ell$, where $x_j \in X$ for every $j\in [\ell]$.

        We show that such a matching can be constructed greedily. To see this, assume for a contradiction that there exists such a matching $x_1u_1,\dots,x_{j}u_{j}$ for some $j< \ell$ which cannot be extended to span $u_{j+1}$. Then, in particular
        \[
            X \subseteq N^+(u_{j+1}) \cup \{x_1,\dots,x_j\},
        \]
        which implies that 
        \[
            |N^+(u_{j+1}) \cap X_i| \geq |X| - j\\
            \geq (1-\tfrac{1}{2}\epsilonb)\Delta - 2b^2-b \\
            > (1-\epsilonb)\Delta + b+1, 
        \]
        where the last inequality holds by~\eqref{eq:largeDelta:3}.
        Since $u_{j+1} \in L$ and, by assumption $L\cap X_i = \emptyset$, this is a contradiction to Claim~\ref{claim:dense_decomposition}\ref{claim:dense_decomposition:iii}. 
    \end{proofclaim}
    
    In the two next subsections, we exhibit the exact structure we make use of in the probabilistic analysis.
    For every $i\in [t]$, we let $\zSCR_i$ denote the set of vertices in $V(D) \setminus X_i$ with less than $\log^4(\Delta)$ in- and out-neighbors in $X_i$, that is
    \[
        \zSCR_i = \left\{ z \in V(D) \setminus X_i :  \max\Big(|N^-(z) \cap X_i|, |N^+(z) \cap X_i|\Big) < \log^4(\Delta) \right\}. 
    \]
    Intuitively, if some vertex $x\in K_i$ has a neighbor $z\in \zSCR_i$ then, in a random partial coloring (which is defined later on), it is likely that $x$ has a neighbor in $K_i$ that receives the same color as $z$, which makes $x$ ``easier'' to color. Hence, if $x$ has sufficiently many such neighbors in $\zSCR_i$, then $x$ will likely be ``easy'' to color in a random coloring. In this case, we will call $x$ a {\it savior}. 
    
    Our goal is thus to show that each $X_i$ has many saviors, implying that, in a random partial coloring, many vertices in $X_i$ are easy to color. We can then extend such a coloring to $X_i$ by keeping those easy vertices for the end.
    
    For technical reasons, the exact definition of a savior for $X_i$ actually depends on the order of $D_i$. 
    In the remainder of the proof, we thus distinguish two cases for $D_i$. We say that $D_i$ is {\it tight} if $|X_i| \geq \Delta-3b+2$, and that it is {\it loose} otherwise (that is, $|X_i| \leq \Delta-3b+1$).

    \subsection{Structure of the loose dense sets}

    We start with the case of loose $D_i$'s, which is a good warm-up before the more technical case of tight $D_i$'s.
    For every $i\in [t]$, we say that a vertex $x$ is a {\it loose $i$-savior} if $x\in K_i$ and 
    \[
        |N^+(x) \cap \zSCR_i| \geq d^+(x) - \Delta+b.
    \]
    The following shows the precise structure we need for loose $D_i$'s.
    
    \begin{claim}
        \label{claim:nb_loose_saviors}
        For every $i\in [t]$, if $D_i$ is loose then there exist at least $\frac{1}{2}\Delta$ distinct loose $i$-saviors.
    \end{claim}
    \begin{proofclaim}
        Let $i\in [t]$ be such that $D_i$ is loose.
        We prove that all vertices in $K_i$ are loose $i$-saviors except at most $b-1$ of them, hence showing that the number of loose $i$-saviors is at least 
        \begin{align*}
            |K_i| - b + 1 &\geq |X_i|-5b+1 &\text{by Claim~\ref{claim:matching_Di},}\\
            &\geq \Delta - \tfrac{4}{\epsilonb}d - 5b +1 &\text{by Claim~\ref{claim:dense_decomposition}\ref{claim:dense_decomposition:i},}\\
            &\geq \tfrac{1}{2}\Delta &\text{by~\eqref{eq:largeDelta:3}.}
        \end{align*}
        Assume for a contradiction that there exist $b$ distinct vertices $y_1,\dots,y_b\in K_i$, none of which is a loose $i$-savior.
        For $j\in [b]$, let $L_j$ denote the set of out-neighbors of $y_j$ outside $X_i \cup \zSCR_i$, that is $L_j = N^+(y_j) \setminus (X_i \cup \zSCR_i)$. 
        Recall that, by definition, since $y_j$ is not a loose $i$-savior, we have $|N^+(y_j) \cap \zSCR_i|  \leq d^+(y_j) - \Delta+b-1$.
        Therefore, we have:
        \begin{align*}
            |L_j| &= |N^+(y_j) \setminus (X_i \cup \zSCR_i)| & \\
            &= d^+(y_j) - |N^+(y_j) \cap X_i| - |N^+(y_j)\cap \zSCR_i| & \\
            &\geq d^+(y_j) - (|X_i|-1) -|N^+(y_j) \cap \zSCR_i| &\text{as $y_j \in X_i \setminus N^+(y_j)$,} \\
            &\geq \Delta-b+2 -|X_i| &\text{as $y_j$ is not a loose $i$-savior,}\\
            &\geq 2b+1 &\text{as $D_i$ is loose.}
        \end{align*}
        For every $j$, let thus $L'_j$ be an arbitrary subset of $L_j$ of size precisely $2b+1$, and let $L= \bigcup_{j\in [b]}L_j'$.
        Let finally $u_1,\dots,u_{\ell}$ be an arbitrary labeling of the vertices in~$L$, where $\ell = |L|$, so $\ell \leq 2b^2+b$.
        
        By Claim~\ref{claim:matching_S}, applied with $X=K_i \setminus\{y_1,\dots,y_b\}$, there exists, in $\overline{D}$, a matching $x_1u_1,\dots,x_{\ell}u_{\ell}$ such that $x_j \in K_i\setminus \{y_1,\dots,y_b\}$ for every $j\in [\ell]$. Note that Claim~\ref{claim:matching_S} can be applied as $|L| \leq 2b^2+b$ and
        \begin{align*}
            |K_i \setminus\{y_1,\dots,y_b\}| &= |X_i| - |U_i| - b \\
            &\geq \Delta-\tfrac{4}{\epsilonb}d -5b &\text{by Claims~\ref{claim:dense_decomposition}\ref{claim:dense_decomposition:i} and~\ref{claim:matching_Di},}\\
            &\geq (1-\tfrac{1}{2}\epsilonb)\Delta &\text{by~\eqref{eq:largeDelta:3}}.
        \end{align*}
            
        Let $\phi$ be a $(\Delta-b+1)$-dicoloring of $D-(K_i\cup L)$, which exists by minimality of $D$. 
        Then, for every $j\in [\ell]$, we extend $\phi$ to $\{x_j,u_j\}$ by choosing for these two vertices a color that is not already appearing in $N^+(x_j) \cup N^+(u_j)$ or in $N^-(x_j) \cup N^-(u_j)$. 
        To see that such a color is indeed available, recall that by construction $u_j \notin \zSCR_i$, which means that $u_j$ has at least $\log^4(\Delta)$ in-neighbors or $\log^4(\Delta)$ out-neighbors in $X_i$.
        Assume that $u_j$ has at least $\log^4(\Delta)$ out-neighbors in $X_i$, the other case being symmetric. 
        Therefore, since the vertices in $X_i\setminus (U_i \cup \{x_1,\dots,x_{j-1}\})$ are uncolored at that step, by Claim~\ref{claim:matching_Di} we get that $u_j$ has at least 
        \[
        \log^4(\Delta) - |U_i| -  (j-1) \geq \log^4(\Delta) -5b - 2b^2
        \]
        uncolored out-neighbors. Moreover, since $x_j\in K_i$, $x_j$ is linked with digons to all vertices in $K_i$, and in particular $x_j$ has at least 
        \[
            |X_i| -|U_i| - (j-1) \geq |X_i|- 5b -2b^2
        \]
        uncolored out-neighbors. 
        It follows that the number of colors appearing in $N^+(x_j) \cup N^+(u_j)$ is at most
        \begin{align*}
            &d^+(x_j) + d^+(u_j) - (\log^4(\Delta) -5b- 2b^2) - (|X_i|-5b-2b^2)&\\
            &\leq 2\Delta- \log^4(\Delta)+ 4b^2 + 12b + 2  - |X_i|  &\text{by Claim~\ref{claim:degrees},}\\
            &\leq \Delta- \log^4(\Delta)+ \tfrac{4}{\epsilonb}d + 4b^2 + 12b + 2  &\text{by Claim~\ref{claim:dense_decomposition}\ref{claim:dense_decomposition:i},}\\
            &\leq \Delta -b, &\text{by~\eqref{eq:largeDelta:3}.}
        \end{align*}
        as desired.
        We then greedily extend the coloring to the vertices in $K_i \setminus \{y_1,\dots,y_b\}$. Note that, by Claim~\ref{claim:degrees}, each of these vertices has in- or out-degree at most $\Delta$, and has at least $b$ in- and out-neighbors that are still uncolored (namely $y_1,\dots,y_b$). Therefore, the number of colors appearing in both the out- and in-neighborhood of every such vertex is at most $\Delta-b$.

        We finally extend the coloring to the remaining uncolored vertices, namely $y_1,\dots,y_b$, by choosing for each of them a color that is not appearing in its out-neighborhood. Recall that $y_j$ dominates $L_j$ (and in particular, it dominates $L_j'$) and $K_i \setminus y_j$. By construction, every color used on $L_j'$ is also used on $K_i\setminus y_j$. Therefore, since $d^+(y_j) \leq \Delta+b+1$ by Claim~\ref{claim:degrees}, the number of colors appearing in the out-neighborhood of $y_j$ is at most 
        \[
            d^+(y_j) - |L_j'| \leq (\Delta+b+1) - (2b+1) = \Delta-b,
        \]
        as desired. This shows that $\phi$ can be extended to $D$, hence showing that $\dic(D) \leq \Delta-b+1$, a contradiction.
    \end{proofclaim}

    \subsection{Structure of the tight dense sets}

    We now move to the more technical case of tight $D_i$'s. The overall proof follows that for loose $D_i$'s, except that it requires a more careful analysis of the number of colors that can be ``saved'' on a savior.

    The reader may note that the structural results of this section also hold for loose $D_i$'s. However, in order to make use of them in the next section, we need to identify certain vertices so as to ensure that they receive the same color via the random coloring process. Such identifications must be carried out carefully, as we need to ensure that the general structure of $D$ is not affected too much. This is guaranteed when identifying two vertices within the same tight $D_i$, but not within a loose $D_i$.

    \medskip

    \begin{figure}
        \centering
        \begin{tikzpicture}
            \node[draw, very thick, minimum width=2cm, minimum height=0.7cm, rectangle, rounded corners=8pt, outer sep=2pt] (Yi) at (0,2) {$Y_i$};
            \node[draw, very thick, minimum width=4cm, minimum height=1.2cm, rectangle, rounded corners=8pt, outer sep=2pt] (Xi) at (0,3.8) {$K_i^\star\setminus Y_i $};
            \node[draw, very thick, minimum width=1cm, minimum height=3cm, rectangle, rounded corners=8pt, outer sep=2pt] (Ri) at (6,2) {$R_i$};
            \foreach \i in {-3,...,3} {
                \node[vertex] (u\i) at (\i*0.6,0.2) {};
                \node[vertex] (v\i) at (\i*0.6,-0.8) {};
                \draw[tedge, orange] (u\i) -- (v\i);
            }
            
            \node[region, thick, fit=(u-3)(v3), label=left:{$U_i^\star$}] (Uis) {};
            
            \draw[digon] (Uis) to (Yi);
            \draw[digon] (Xi) -- (Yi);
            \draw[digon, dashed] (Uis.143) to[out=90, in=-90] (Xi.-155);

            \draw[dotted, thick, gray] (Ri) -- (Yi);
            \draw[dotted, thick, gray] (Ri) to[out=180, in=-30] (Xi);
            \draw[dotted, thick, gray] (Ri) to[out=180, in=40] (Uis);
            \draw[dotted, thick, gray] (Ri.180) to[out=180, in=90] (4.3,1.5);
            \draw[dotted, thick, gray] (4.3,1.5) to[out=-90, in=180] (Ri.240);

            \node[gray] at (4.7,2.3) {$\geq \fb$};
        \end{tikzpicture}
        \caption{An illustration of $R_i$, $U_i^\star$, $K_i^\star$, and $Y_i$. The matching $M_i^\star$ is highlighted in orange. A solid digon between two sets illustrates the presence of all possible such digons. The dashed digon illustrates that vertices in $U_i^\star$ are linked with digons to almost all vertices in $K_i^\star$. The dotted lines illustrate that $|X_i\setminus N^\pm(r)|\geq \fb$ for every $r\in R_i$.}
        \label{fig:partition_Xi}
    \end{figure}
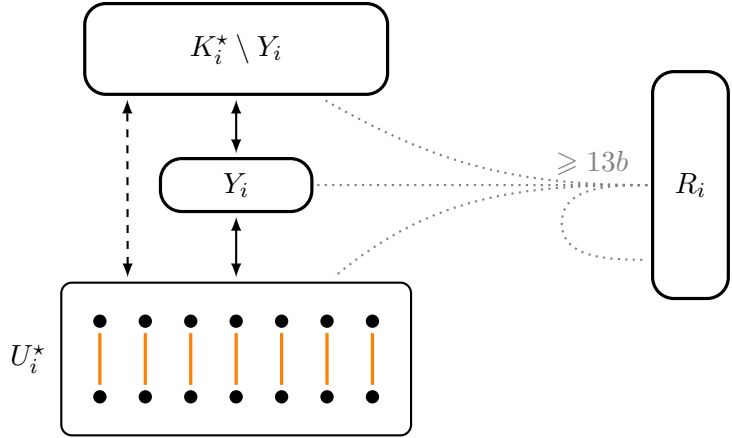

    We start with a few technical definitions, see Figure~\ref{fig:partition_Xi} for an illustration of these.
    For every $i\in [t]$ we let $R_i$ be the set of vertices in $X_i$ that miss $13b$ digons in $X_i$, that is
    \[
        R_i = \{x\in X_i : |X_i\setminus N^\pm(x)|  \geq \fb\}.
    \]
    Note that $R_i \subseteq U_i$: since $K_i$ forms a biclique, every vertex $x\in K_i$ satisfies $|N^\pm(x)\setminus X_i| \leq |U_i|$ while $|U_i| < \fb$ by Claim~\ref{claim:matching_Di}.
    In particular, since $R_i \subseteq U_i$ then $|R_i| \leq 4b$ by Claim~\ref{claim:matching_Di}. 
    
    Then, we let $M_i^\star$ be maximum matching of $\overline{D}[X_i \setminus R_i]$, we let $\nu_i^\star = |M_i^\star|$, and we let $U_i^\star$ be the set of vertices spanned by $M_i^\star$. Observe that $M_i^\star$ is also a matching of $\overline{D_i}$, so by Claim~\ref{claim:matching_Di},
    \[
        |U_i^\star| = 2\nu_i^\star \leq 2\nu_i \leq 4b.
    \]
    We further let $K_i^\star = X_i \setminus (R_i \cup U_i^\star)$. Observe that, by definition of $M_i^\star$, $K_i^\star$ is a biclique.
    Finally, we let 
    \[
        Y_i = K_i^\star \cap \bigcap_{u\in U_i^\star} N^\pm(u).
    \]
    By construction, we have the following three claims that will be useful in the probabilistic analysis later.
    
    \begin{claim}
    	\label{claim:size_K}
    	For every integer $i\in [t]$, $|K_i| \geq \Delta-\tfrac{4}{\epsilonb}d - 4b$ and $|K_i^\star| \geq \Delta-\tfrac{4}{\epsilonb}d - 8b$.
    \end{claim}
    \begin{proofclaim}
    	By definition, $K_i = X_i \setminus U_i$, and $K_i^\star = X_i \setminus (U_i^\star \cup R_i)$. Since $R_i\subseteq U_i$, and both $U_i$, $U_i^\star$ have size at most $4b$ by Claim~\ref{claim:matching_Di}, the result follows from Claim~\ref{claim:dense_decomposition}\ref{claim:dense_decomposition:i}. 
    \end{proofclaim}
    
    \begin{claim}
        \label{claim:adjacency_Uis}
        For every $i\in [t]$ and every $u\in U_i^\star$, we have $|X_i\setminus N^\pm(u)| < \fb$ and $Y_i \subseteq N^\pm(u)$. Moreover, for every $i\in [t]$, $|Y_i| \geq \frac{3}{4}\Delta$.
    \end{claim}
    \begin{proofclaim}
        We have that $|N^\pm(u)\setminus X_i| < \fb$ as $U_i^\star$ is chosen explicitly in $X_i \setminus R_i$, and $Y_i \subseteq N^\pm(u)$ by definition.
        For the last inequality, we have
        \begin{align*}
            |Y_i| &\geq |K_i^\star| - \fb \cdot |U_i^\star|\\
            &\geq \Delta -\tfrac{4}{\epsilonb}d- 8b - 52b^2 &\text{by Claims~\ref{claim:size_K} and~\ref{claim:matching_Di},}\\
            &\geq\tfrac{3}{4}\Delta &\text{by~\eqref{eq:largeDelta:3}.}
        \end{align*}
        The claim follows.
    \end{proofclaim}

    \begin{claim}
        \label{claim:adjacency_Uis_bis}
        For every $i\in [t]$ and every $uv\in M_i^\star$, if $D_i$ is tight then $|N^\pm(u)\cap N^\pm(v)| \geq \Delta-29b+2$.
    \end{claim}
    \begin{proofclaim}
        We have 
        \begin{align*}
            |N^\pm_D(u) \cap N^\pm_D(v)| &\geq |X_i| - |X_i \setminus N^\pm(u)|- |X_i \setminus N^\pm(v)|\\
            &> |X_i| - 26b &\text{by Claim~\ref{claim:adjacency_Uis},}\\
            &\geq \Delta -29b+2 &\text{as $D_i$ is tight.}
        \end{align*}
        The claim follows.
    \end{proofclaim}

    We are finally ready to introduce the definition of saviors. We say that a vertex $x\in X_i$ is a {\it tight $i^+$-savior} if each of the following holds:
    \begin{itemize}
        \item $x\in Y_i$,
        \item $d^+(x)\leq \Delta$, and
        \item $|N^+(x) \cap \zSCR_i| \geq d^+(x) - \Delta +b-\nu_i^\star$.
    \end{itemize} 
    Similarly, we say that $x\in X_i$ is a {\it tight $i^-$-savior} if each of the following holds:
    \begin{itemize}
        \item $x\in Y_i$,
        \item $d^-(x)\leq \Delta$, and
        \item $|N^-(x) \cap \zSCR_i| \geq d^-(x) - \Delta +b-\nu_i^\star$.
    \end{itemize} 
    Finally, we say that $x\in X_i$ is a {\it tight $i$-savior} if it is a tight $i^+$-savior or a tight $i^-$-savior. We conclude this part with the following key structural result.

    \begin{claim}
        \label{claim:nb_tight_saviors}
        For every $i\in [t]$, there exist at least $\frac{1}{2}\Delta$ distinct tight $i$-saviors.
    \end{claim}
    \begin{proofclaim}
        We prove that all vertices in $Y_i$ are tight $i$-saviors except at most $b-1$ of them, hence showing that the number of tight $i$-saviors is at least 
        \[
            |Y_i| - b+1 \geq \tfrac{3}{4}\Delta - b+1 \geq \tfrac{1}{2}\Delta,
        \]
        where first inequality follows from Claim~\ref{claim:adjacency_Uis}, and the second one from~\eqref{eq:largeDelta:3}. Assume for a contradiction that there exist $b$ distinct vertices $y_1,\dots,y_b\in Y_i$, none of which is a tight $i$-savior. 
        For every $j\in [b]$, let
        \[
            N_j =  
            \begin{cases}
                N^+(y_j) & \text{if $d^+(y_j) \leq \Delta$}\\
                N^-(y_j) & \text{otherwise.}
            \end{cases}
        \]
        Observe that $|N_j| \leq \Delta$ by Claim~\ref{claim:degrees}. Further, for every $j\in [b]$, let 
        \[
            L_j = N_j \setminus (X_i \cup \zSCR_i).
        \]
        Note also that, by definition, since $y_j$ is not a tight $i$-savior, we have 
        \[
        |N_j \cap \zSCR_i|  \leq |N_j| -\Delta + b-\nu^\star_i-1.
        \]
        Therefore, we can lower bound the size of $L_j$ as follows:
        \begin{align*}
            |L_j| &= |N_j \setminus (X_i \cup \zSCR_i)| & \\
            &= |N_j| - |N_j \cap K_i^\star| - |N_j\cap U_i^\star| -|N_j\cap \zSCR_i|- |N_j\cap R_i| & \\
            &= |N_j| - |K_i^\star| + 1 - 2\nu_i^\star - |N_j\cap \zSCR_i|- |N_j\cap R_i|&\text{as $N_j \cap K_i^\star = K_i^\star \setminus \{y_j\}$,} \\
            &\geq - |K_i^\star| +2+ \Delta-b - \nu_i^\star- |N_j\cap R_i| &\text{as $y_j$ is not a tight $i$-savior,}\\
            &\geq b-\nu_i^\star- |N_j\cap R_i| &\text{as $K_i^\star$ is a biclique,}
        \end{align*}
        where in the last inequality we further use the assumption that $\bic(D) \leq \Delta-2b+2$.
        For every $j$, let thus $L_j'$ be an arbitrary subset of $L_j$ of size precisely $b-\nu_i^\star- |N_j\cap R_i|$, and let $L= \bigcup_{j\in [b]}L_j'$. Let $\ell = |L|$ and observe that $\ell \leq b^2$.
        Let finally $u_1,\dots,u_{\ell}$ be an arbitrary labeling of the vertices in $L$.

        Let $M_R$ be a matching between $R_i$ and $K_i^\star \setminus \{y_1,\dots,y_b\}$ in $\overline{D}$. Note that such a matching exists, as it can be constructed greedily. Indeed, for every $r\in R_i$, we have
        \begin{align*}
            |(K_i^\star \setminus \{y_1,\dots,y_b\}) \setminus N^\pm(r)| &\geq |X_i\setminus N^\pm(r)| - |U_i^\star| - |R_i| - b\\
            &\geq \fb - 4b - 4b- b &\text{by definition of $R_i$,}\\
            &\geq 4b\\
            &\geq |R_i|.
        \end{align*}
        Let $U_R$ be the set of vertices spanned by $M_R$.
        By Claim~\ref{claim:matching_S} applied with 
        \[
        X= K_i^\star \setminus (U_R \cup \{y_1,\dots,y_b\}),
        \]
        there exists, in $\overline{D}$, a matching $\{x_1u_1,\dots x_\ell u_\ell\}$, vertex-disjoint from $M_R$, such that $x_j \in K_i^\star \setminus \{y_1,\dots,y_b\}$ for every $j\in [\ell]$. Note that it is possible to apply Claim~\ref{claim:matching_S} because $|L|\leq b^2$ and
        \begin{align*}
            |K_i^\star \setminus (U_R \cup \{y_1,\dots,y_b\})| &= |K_i^\star| - |R_i| - b \\
            &\geq \Delta-\tfrac{4}{\epsilonb}d - 13b&\text{by Claim~\ref{claim:size_K},}\\
            &\geq (1-\tfrac{1}{2}\epsilonb)\Delta &\text{by~\eqref{eq:largeDelta:3}}.
        \end{align*}
        
        Let $\phi$ be a $(\Delta-b+1)$-dicoloring of $D-((X_i \cup L) \setminus (U_i^\star \cup U_R) )$, with the extra property that $\phi(u) = \phi(v)$ for every $uv \in M_i^\star \cup M_R$. Note that such a dicoloring exists by Claims~\ref{claim:extending_matching} and~\ref{claim:matching_Di} and the fact that $M_i^\star \cup M_R$ is a matching of $\overline{D_i}$ (up to uncoloring the vertices in $L$).

        For every $j\in [\ell]$, we extend $\phi$ to $\{x_j,u_j\}$ by choosing for these two vertices a color that is not already appearing in $N^+(x_j) \cup N^+(u_j)$ or in $N^-(x_j) \cup N^-(u_j)$. 
        To see that such a color is available, recall that by construction $u_j \notin \zSCR_i$, which means that $u_j$ has at least $\log^4(\Delta)$ in-neighbors or $\log^4(\Delta)$ out-neighbors in $X_i$.
        Assume that $u_j$ has at least $\log^4(\Delta)$ out-neighbors in $X_i$, the other case being symmetric. 
        Therefore, since the vertices in $X_i\setminus (U_i^\star \cup U_R\cup  \{x_1,\dots,x_{j-1}\})$ are uncolored at that step, we get that $u_j$ has at least 
        \[
        \log^4(\Delta) - |U_i^\star| - |U_R| -  (j-1) \geq \log^4(\Delta) -4b - b^2
        \]
        uncolored out-neighbors, where in the inequality we use Claim~\ref{claim:matching_Di} and the fact that $M_i^\star \cup M_R$ is a matching of $\overline{D_i}$. Moreover, since $x_j\in K_i^\star$, $x_j$ is linked with digons to all vertices in $K_i^\star$, and in particular $x_j$ has at least 
        \[
            |K_i^\star| - |U_R\cap K_i^\star| -  (j-1) \geq |K_i^\star|-2b - b^2
        \]
        uncolored out-neighbors. 
        It follows that the number of colors appearing in $N^+(x_j) \cup N^+(u_j)$ is at most
        \begin{align*}
            &d^+(x_j) + d^+(u_j) - \log^4(\Delta)- |K_i^\star| + 6b +2b^2\\
            &\leq \Delta- \log^4(\Delta) + \tfrac{4}{\epsilonb}d+  2+16b + 2b^2  &\text{by Claims~\ref{claim:degrees} and~\ref{claim:size_K}}\\
            &\leq \Delta -b, &\text{by~\eqref{eq:largeDelta:3}}
        \end{align*}
        as desired.
        We then greedily extend the coloring to the uncolored vertices in $K_i^\star \setminus  \{y_1, \dots, y_b\}$. Note that, by Claim~\ref{claim:degrees}, each of these vertices has in- or out-degree at most $\Delta$, and has at least $b$ in- and out-neighbors that are still uncolored (namely $y_1,\dots,y_b$). Therefore, the number of colors appearing in both the out- and in-neighborhoods of every such vertex is at most $\Delta-b$.

        We finally extend the coloring to the remaining uncolored vertices, namely $y_1,\dots,y_b$, by choosing for $y_j$ a color that is not appearing in $N_j$.
        Recall that, by construction, we have:
        \begin{itemize}
            \item $U_i^\star\subseteq N_j$,
            \item $L_j' \subseteq N_j$, and
            \item $U_R \cap X_i^\star \subseteq N_j$. 
        \end{itemize}
        Therefore, by construction of the current partial dicoloring, the number of colors appearing in $N_j$ is at most 
        \[
            |N_j| - \nu_i^\star - |L_j'| - |N_j \cap R_i| \leq |N_j| -b \leq \Delta-b
        \]
        by choice of $L_j'$ and the fact that $|N_j|\leq \Delta$. This shows that $\phi$ can be extended to $D$, hence showing that $\dic(D) \leq \Delta-b+1$, a contradiction.
    \end{proofclaim}

    \subsection{The random coloring process}

    The goal of this subsection is to define a specific random coloring process. In the next subsection, using probabilistic methods, we leverage on the structural properties obtained previously to show that, with positive probability, the partial dicoloring obtained by this process can be extended to $D$, hence contradicting $\dic(D) > \Delta-b+1$.
    
    Given a partial coloring $\psi$ of a digraph $H$, we denote by $\dom(\psi)$ its domain. For any subset of vertices $X\subseteq V(H)$, we denote by $\psi(X)$ the set of colors given by $\psi$ on $X$, that is
    \[
    \psi(X) = \{ \psi(x) : x\in X\cap \dom(\psi)\}.
    \]
    
    From now on, let $\iCAL$ be the set of indices $i\in [t]$ such that $D_i$ is tight.
    Let $M^\star = \bigcup_{i\in \iCAL}M_i^\star$. Let $D^\star$ be the digraph obtained from $D$ by identifying each $\{u,v\}\in M^\star$ into a single vertex $x_{u,v}$.
    Assign to each vertex $v\in V(D^\star)$ a color from $[1,\lceil \Delta/2 \rceil]$ drawn independently and uniformly at random, and let $\phi'$ be the obtained coloring. Note that $\lceil\Delta/2\rceil < \Delta-b+1$ by~\eqref{eq:largeDelta:3}. Then, we simultaneously uncolor all the vertices $v\in V(D^\star)$ whose color appear in both the in-neighborhood and the out-neighborhood of $v$, that is,s such that
    \[
    \phi'(v) \in \phi'\Big(N^-_{D^\star}(v)\Big) \cap \phi'\Big(N^+_{D^\star}(v)\Big),
    \]
    and we let $\phi^\star$ be the resulting partial coloring. Note that $\phi^\star$ is a partial $(\Delta-b+1)$-dicoloring of $D^\star$, as any vertex is a sink or a source in its color class. Finally, let $\phi$ be the partial coloring of $D$ naturally derived from $\phi^\star$, that is 
    \begin{align*}
    	\dom(\phi) = 
    	&~\Big\{v \in V(D) \cap V(D^\star) : v\in \dom(\phi^\star)\Big\}\\
    	&\cup \Big\{u,v : \{u,v\}\in M^\star \text{~and~} x_{u,v} \in \dom(\phi^\star)\Big\},
    \end{align*}
    and for every $v\in \dom(\phi)$,
    \[
    \phi(v) = 
    \begin{cases}
    	\phi^\star(v) & \text{if $v\in V(D^\star)$}\\
    	\phi^\star(x_{u,v}) & \text{otherwise, where $u$ is such that $\{u,v\} \in M^\star$.}
    \end{cases}
    \]
    Note that $\phi$ is a partial $(\Delta-b+1)$-dicoloring of $D$ by Lemma~\ref{lemma:coloring_contraction}.
    For every vertex $s\in S$, we say that $\phi$ is {\it $s$-extendable} if at least one of the following holds:
    \begin{itemize}
        \item $\big|M^\star \cap \binom{N^+_D(s)}{2}\big| \geq 2b+1$, or
        \item $\big|\dom(\phi) \cap N^+_D(s)\big| - \big|\phi\big(N^+_D(s)\big)\big| \geq 2b+1$.
    \end{itemize}
    Note that the first property does not depend on $\phi$.

    \medskip
    
    For every $i\in [t]$, we say that a vertex $x$ is a {\it loose $i$-rescuer} (with respect to $\phi$) if $x\in K_i$, $x$ is uncolored (that is, $x\notin \dom(\phi)$), and 
    \[
    \big|\dom(\phi) \cap N^+_D(x)\big| - \big|\phi\big(N^+_D(x)\big)\big| \geq d^+_D(x)-\Delta +b.
    \]
    We say that $\phi$ is {\it loosely $i$-extendable} if there exist at least $b$ distinct loose $i$-rescuers.

    \medskip
    
    Similarly, a vertex $x$ is a {\it tight $i$-rescuer} if $x\in Y_i$, $x$ is uncolored, and at least one of the following holds:
    \begin{itemize}
        \item $d^+_D(x) \leq \Delta$ and $\big|\dom(\phi) \cap N^+_D(x) \setminus U_i^\star\big| - \big|\phi\big(N^+_D(x) \setminus U_i^\star\big)\big| \geq d^+_D(x)-\Delta +b - \nu_i^\star$; or
        \item $d^-_D(x) \leq \Delta$ and $\big|\dom(\phi) \cap N^-_D(x) \setminus U_i^\star\big| - \big|\phi\big(N^-_D(x) \setminus U_i^\star\big)\big| \geq d^-_D(x)-\Delta +b-\nu_i^\star$.
    \end{itemize}
    We further say that $\phi$ is {\it tightly $i$-extendable} if there exist at least $b$ tight $i$-rescuers. 

    \medskip
    
    We say that $\phi$ is {\it $i$-extendable if}
    \begin{itemize}
        \item $D_i$ is loose and $\phi$ is loosely $i$-extendable; or
        \item $D_i$ is tight and $\phi$ is tightly $i$-extendable.
    \end{itemize}

    Finally, we say that $\phi$ is {\it extendable} if it is $s$-extendable for every $s\in S$ and $i$-extendable for every $i\in [t]$.
    We conclude this section by showing that $\phi$ cannot be extendable.

    \begin{claim}
        \label{claim:no_extendable_dicoloring}
        If $\phi$ is extendable, then $\dic(D) \leq \Delta-b+1$.
    \end{claim}
    \begin{proofclaim}
        We show that $\phi$ can be extended to $D$.
        We first extend $\phi$ to $X_i$ for every $i\in [t]$ as follows. 
        Recall that $\phi$ uses at most $\lceil \Delta/2\rceil$ colors, which in particular implies that $|\dom(\phi) \cap X_i| \leq \frac{2}{3}\Delta$. Indeed, if this is not the case, then
        \[
        \tfrac{2}{3}\Delta < |\dom(\phi) \cap X_i| \leq |\dom(\phi)\cap K_i| + |U_i| \leq |\dom(\phi)\cap K_i| + 4b,
        \]
        which implies that at least $\frac{2}{3}\Delta - 4b > \lceil \frac{1}{2}\Delta\rceil $ vertices of $K_i$ are colored by $\phi$. In particular, two vertices of $K_i$ are colored identically in $\phi$. Since $K_i$ is a biclique, this is a contradiction to $\phi$ being a partial dicoloring. 
        
        Assume first that $D_i$ is loose, so $\phi$ is loosely $i$-extendable, and let $\lSCR$ be a set of $b$ loose $i$-rescuers. We first extend the dicoloring to each vertex $u\in U_i \setminus \dom(\phi)$ by a color that is not appearing in $N^+(u)$. This is possible as the number of colors appearing in $N^+_D(u)$ is at most
        \begin{align*}
            &|\dom(\phi) \cap X_i| + |N^+_S(u) \setminus X_i| + |U_i|\\
            &\leq \tfrac{2}{3}\Delta + d^+_D(u) - |N^+_D(u) \cap X_i| + |U_i|\\
            &\leq (\tfrac{2}{3}+\epsilonb)\Delta +6b+1 &\text{by Claims~\ref{claim:degrees}, \ref{claim:dense_decomposition}\ref{claim:dense_decomposition:ii}, and~\ref{claim:matching_Di},}\\
            &\leq \Delta-b &\text{by choice of $\epsilonb$ and~\eqref{eq:largeDelta:3}.}
        \end{align*}
        We then extend the coloring to $K_i \setminus (\lSCR\cup \dom(\phi))$. This is possible because each of these vertices has in- or out-degree at most $\Delta$, and at least $b$ uncolored in- and out-neighbors, as $\lSCR \subseteq K_i$ and $K_i$ is a biclique. Note that the vertices in $\lSCR$ are indeed uncolored by definition of loose $i$-rescuers.
        We can finally extend the coloring to each vertex $y\in \lSCR$ by choosing a color that is not used on $N^+_D(y)$. To see that this is possible, note that $y$ being a loose $i$-rescuer implies that there exists a set $B \subseteq N^+_D(y)$ of $d^+_D(y)-\Delta+b$ vertices whose colors in $\phi$ appear in $N^+_D(y) \setminus B$. Hence, the number of colors appearing in the out-neighborhood of $y$ is at most 
        \[
            d^+(y) - (d^+(y)-\Delta+b) = \Delta-b,
        \]
        as desired.

        Assume now that $D_i$ is tight, so $\phi$ is tightly $i$-extendable, and let $\lSCR$ be a set of $b$ tight $i$-rescuers.
         We proceed as above, but here we need to ensure that the vertices matched in $M_i^\star$ receive the same color. Note that this is the case for the vertices of $U_i^\star$ that are already colored,
         as by construction of $\phi$, for every $uv\in M_i^\star$, since $u$ and $v$ are identified in $D^\star$, either $\{u,v\} \cap \dom(\phi) = \emptyset$ or $\{u,v\} \subseteq \dom(\phi)$ and $\phi(u) = \phi(v)$.
         For every $uv\in M_i^\star$ such that $\{u,v\} \cap \dom(\phi) =\emptyset$, we extend the dicoloring by giving to both $u$ and $v$ a color that is not appearing in $N^+_D(u) \cup N^+_D(v)$. This is possible as the number of colors appearing in $N^+_D(u)\cup N^+_D(v)$ is at most
        \begin{align*}
            |\dom(\phi) \cap X_i| + |N^+_D(u) \setminus X_i|+ |N^+_D(v) \setminus X_i| + |U_i^\star| \leq (\tfrac{2}{3}+2\epsilonb)\Delta +8b+2 \leq \Delta-b.
        \end{align*}
        by~\eqref{eq:largeDelta:3}. We then extend the dicoloring to $R_i \setminus \dom(\phi)$, which is possible as the number of colors appearing in $N^+(r)$ is at most  
        \begin{align*}
            |\dom(\phi) \cap X_i| + |N^+(r) \setminus X_i|+ |U_i^\star| + |R_i| \leq (\tfrac{2}{3}+\epsilonb)\Delta +10b +1\leq \Delta-b
        \end{align*}
        by~\eqref{eq:largeDelta:3}.  We then extend the coloring to $K_i^\star \setminus (\lSCR\cup \dom(\phi))$, which is again possible because each of these vertices has $b$ uncolored out- and in-neighbors as $K_i^\star$ is a biclique. We finally extend the coloring to each vertex $y\in \lSCR$. To see that this is possible, recall that, by definition of being a tight $i$-rescuer, $y\in Y_i$ and that, by Claim~\ref{claim:adjacency_Uis}, $U_i^\star \subseteq N^\pm_D(y)$. Note also that, as $y$ is a tight $i$-rescuer, at most $\Delta-b+\nu_i^\star$ colors appear in both 
        $N^+_D(y) \setminus U_i^\star$ and $N^-_D(y) \setminus U_i^\star$.
        By construction of the current partial dicoloring, we finally obtain that the maximum number of colors appearing in $N^+_D(y)$ and $N^-_D(y)$ is at most 
        \[
            (\Delta-b+\nu_i^\star)-\tfrac{1}{2}|U_i^\star| = \Delta-b,
        \]
        as desired.

        We finally extend $\phi$ to $S$ by choosing for every vertex $s\in S$ a color that is not appearing in $N^+_D(s)$. Let us justify that such a color is available.
        First, if $|M^\star \cap \binom{N^+_D(s)}{2}| \geq 2b+1$, then since we explicitly gave the same color to the vertices matched in $M^\star$, the number of colors used on $N^+_D(s)$ is at most $d^+_D(s) - (2b+1) \leq \Delta-b$
        by Claim~\ref{claim:degrees}, as desired. Otherwise, if $|M^\star \cap \binom{N^+_D(s)}{2}| < 2b+1$, note that by definition of $\phi$ being $s$-extendable, there exists a set $B_s \subseteq N^+_D(s)$ of $2b+1$ vertices whose colors in $\phi$ appear in $N^+_D(s) \setminus B_s$. Hence, the number of colors appearing in the out-neighborhood of $s$ is again at most $d^+_D(s) - (2b+1) \leq \Delta-b$. 
    \end{proofclaim}

    \subsection{The probabilistic analysis}

    Our final goal is to prove that 
    \[
        \PP(\phi \text{ is extendable}) >0,
    \]
    which together with Claim~\ref{claim:no_extendable_dicoloring} implies that $\dic(D) \leq \Delta-b+1$, a contradiction. 
    To prove this, we show that, for every $s\in S$, $\phi$ is likely to be $s$-extendable, and that for every $i\in [t]$, $\phi$ is likely to be $i$-extendable. We finally conclude that $\phi$ is extendable with positive probability using Lov\'asz Local Lemma.
    
    We make use of the following observation, intuitively saying that the maximum degree in $D^\star$ is still close to $\Delta$.
    
    \begin{claim}
        \label{claim:degrees_Ds}
        Every vertex $x\in V(D^\star)$ satisfies $d^+_{D^\star}(x) \leq \Delta+31b$ and $d^-_{D^\star}(x) \leq \Delta+31b$.
    \end{claim}
    \begin{proofclaim}
        If $x\in V(D^\star)\cap V(D)$, then this is an immediate consequence of Claim~\ref{claim:degrees}, as then identifying vertices cannot increase the degree of $x$.
        Assume thus that $x=x_{u,v}$ corresponds to the identification of $\{u,v\} \in M^\star$. Recall that, by definition of $M^\star$, there exists some $i\in \iCAL$ such that $\{u,v\} \in M_i^\star$ (recall that $\iCAL$ denotes the set of integers $i$ such that $D_i$ is tight). Therefore, we have
        \begin{align*}
            d^+_{D^\star}(x) &\leq d^+_{D}(v) + |N^+_D(u)\setminus N^+_D(v)|\\
            &\leq d^+_D(v) + d^+_D(u) - |N^\pm_D(u) \cap N^\pm_D(v)|\\
            &\leq 2(\Delta+b+1) - |N^\pm_D(u) \cap N^\pm_D(v)| &\text{by Claim~\ref{claim:degrees},}\\
            &\leq \Delta+31b &\text{by Claim~\ref{claim:adjacency_Uis_bis}},
        \end{align*}
        and similarly $d^-_{D^\star}(x) \leq \Delta+31b$,
        as desired.
    \end{proofclaim}

    \subsubsection{Probabilistic analysis: the sparse vertices}

    We first prove that, for every $s\in S$,  $\phi$ is likely to be $s$-extendable. 
    We make use of the following observation, which intuitively says that the vertices in $S$ remain sparse in $D^\star$. Note that $S$ is not spanned by $M^\star$ by construction, so indeed $S\subseteq V(D^\star)$. 

    \begin{claim}
        \label{claim:preserve_sparseness}
        For every vertex $s\in S$, the digraph $D^\star[N^+_{D^\star}(s)]$ contains at most $\Delta^2 - \frac{1}{3}d\Delta$ arcs.
    \end{claim}
    \begin{proofclaim}
        Let $m(s)$ and $m^\star(s)$ denote the number of arcs in $D[N^+_D(s)]$ and $D^\star[N^+_{D^\star}(s)]$ respectively.
        Recall that $\delm$ denotes $\delmax(D)$.
        By Claim~\ref{claim:dense_decomposition}\ref{claim:dense_decomposition:iv} and by definition of being $d/2$-sparse, we have 
        \[
        m(s) \leq \delm(\delm-1) - \tfrac{1}{2}d\delm.
        \]
        
        Observe first that identifying two vertices in $N^+_D(s)$ or two vertices in $V(D)\setminus (N^+_D(s)\cup \{s\})$ does not increase the number of arcs induced by the out-neighborhood of $s$. Let $u_1v_1,\dots,u_rv_r$ be the edges of $M^\star$ with exactly one extremity in $N^+_D(s)$. Assume without loss of generality that $u_j \in N^+_D(s)$ and $v_j \notin N^+_D(s)$ for every $j\in [r]$. 

        Next, observe that if some arc $a$ belongs to $D^\star[N^+_{D^\star}(s)]$ but not to $D[N^+_D(s)]$, necessarily there is some $j\in [r]$ such that $a$ is incident to $v_j$ but not to $u_j$ in $D$. Recall that $u_jv_j \in M_i^\star$ for some $i$ such that $D_i$ is tight. Therefore, 
        \begin{align*}
            m^\star(s) &\leq m(s) + \sum_{j\in [r]} \Big(\big|N^+_{D}(v_j) \setminus N^+_D(u_j)\big| + \big|N^-_{D}(v_j) \setminus N^-_D(u_j)\big|\Big)\\
            &\leq m(s) + \sum_{j\in [r]}\Big( d^+_D(v_j) +d^-_D(v_j) - 2|N^\pm_D(u_j) \cap N^\pm_D(v_j)| \Big)\\
            &\leq m(s) + 60b\cdot r &\text{by Claims~\ref{claim:degrees} and~\ref{claim:adjacency_Uis_bis},}\\
            &\leq m(s) + 60b\cdot \delm\\
            &\leq \delm(\delm-1) - (\tfrac{1}{2}d-60b)\delm \\
            &\leq (\Delta+b+1)(\Delta+b) - (\tfrac{1}{2}d-60b)(\Delta-b) &\text{by Claim~\ref{claim:degrees},} \\
            &\leq \Delta^2 - \tfrac{1}{3}d\Delta &\text{by~\eqref{eq:largeDelta:5},}
        \end{align*}
        as desired.
    \end{proofclaim}
    
    \begin{claim}
        \label{claim:prob_s_extendable}
        For every vertex $s\in S$, $\PP(\text{$\phi$ is not $s$-extendable})\leq \exp(-\log^2(\Delta))$.
    \end{claim}
    \begin{proofclaim}
        Let us fix a vertex $s\in S$. 
        Let $M_s$ be the edges of $M^\star$ included in $N^+_D(s)$, that is $M_s = M^\star \cap \binom{N^+(s)}{2}$.
        We assume that $|M_s| \leq 2b$, for otherwise $\phi$ must be $s$-extendable by definition (that is, $\PP(\text{$\phi$ is $s$-extendable}) = 1$). 
        In particular, we have
        \begin{equation}
            \label{eq:mindeg_sparse}
            d^+_{D^\star}(s) \geq d^+_{D}(s) - 2b \geq \Delta-3b+1.
        \end{equation}
        by Claim~\ref{claim:degrees}.
        For the sake of better readability, in the following let us denote by $N_s$ the set of out-neighbors of $s$ in $D^\star$. 
        Let $Z_s$ denote the number of colors $c\in [\lceil\Delta/2\rceil]$ such that $c$ is used and retained by exactly two vertices in $N_s$, and not used by any other vertex in $N_s$. Formally, $Z_s$ is the number of colors $c\in [\lceil \Delta/2\rceil ]$ for which there exists distinct vertices $x,y\in N_s$ such that:
        \begin{itemize}
            \item $\phi'(x) = \phi'(y) = c$,
            \item $\phi'(z) \neq c$ for every $z\in N_s \setminus \{x,y\}$, and
            \item $\{x,y\} \subseteq \dom(\phi^\star)$.
        \end{itemize}
        Observe that 
        \[
            |\dom(\phi) \cap N^+_D(s) | - |\phi(N^+_D(s))| \geq |\dom(\phi^\star) \cap N^+_{D^\star}(s) | - |\phi(N^+_{D^\star}(s))|\geq Z_s,
        \]
        which implies that
        \[
            \PP(\text{$\phi$ is $s$-extendable}) \geq \PP(Z_s\geq 2b+1) = \PP(Z_s> 2b).
        \]
        We bound the latter term by first proving that the expectation of $Z_s$ is large and then that $ Z_s$ is concentrated around its expectation. 
        Let $\bSCR_s$ be the set of pairs $\{u,v\} \subseteq N_s$ such that $D^\star$ contains at most one arc between $u$ and $v$.
        
        \begin{subclaim}
            \label{subclaim:sparse_1}
            We have $\EE(Z_s) \geq \frac{1}{2^{17}\Delta}|\bSCR_s|$.
        \end{subclaim}
        \begin{proofsubclaim}
            For every pair of vertices $\{x,y\} \in \bSCR_s$ and every color $c\in [\lceil\Delta/2\rceil]$, we let $A_{\{x,y\},c}$ be the event that $x,y$ are the unique vertices colored $c$ in $N_s$ and that they both retain their color. We let $Z_{\{x,y\},c}$ be the binary random variable equals to $1$ if $A_{\{x,y\},c}$ holds and $0$ otherwise. By definition, we have 
            \[
                Z_s = \sum_{\substack{\{x,y\}\in \bSCR_s,\\c\in [\lceil\Delta/2\rceil]}} Z_{\{x,y\},c}.
            \]
            Note that $A_{\{x,y\},c}$ holds in particular if $\phi'(x) = \phi'(y) = c$ and $\phi'(w)\neq c$ for every $w\in (N_s \cup N^+_{D^\star}(x)\cup N^+_{D^\star}(y))\setminus \{x,y\}$. Therefore, using that $\frac{\Delta}{2}\leq \lceil \frac{\Delta}{2}\rceil \leq \Delta$ we have
            \begin{align*}
                \EE(Z_{\{x,y\},c}) = \PP(A_{\{x,y\},c}) &\geq \left(\frac{1}{\Delta}\right)^2\cdot \left(1-\frac{2}{\Delta}\right)^{3\Delta+93b} &\text{by Claim~\ref{claim:degrees_Ds},}\\
                &\geq \frac{1}{\Delta^2}\cdot \left(1-\frac{2}{\Delta}\right)^{4\Delta} &\text{by~\eqref{eq:largeDelta:3},}\\
                &\geq \frac{1}{2^{16}\Delta^2}&\text{by~\eqref{eq:largeDelta:6}.}
            \end{align*}
            By linearity of the expectation, it follows that 
            \[
                \EE(Z_s) \geq \frac{\Delta}{2}\cdot |\bSCR_s| \cdot \frac{1}{2^{16}\Delta^2} = \frac{|\bSCR_s|}{2^{17}\Delta},
            \]
            as desired.
        \end{proofsubclaim}
        \begin{subclaim}
            \label{subclaim:sparse_2}
            For any real number $\lambda> 714\sqrt{|\bSCR_s|/\Delta} + 2752$, we have
            \[
            \PP(|Z_s-\EE(Z_s)|>\lambda)\leq 8\exp\left(\frac{-\lambda^2\Delta}{1024|\bSCR_s| + 256\lambda\Delta}\right).
            \]
        \end{subclaim}
        \begin{proofsubclaim}
            We consider two auxiliary random variables.
            Let $\Zsp$ be the number of colors $c$ such that, for some $\{x,y\} \in \bSCR_s$, $\phi'(x) = \phi'(y) =c$.
            Let $\Zsm$ be the number of colors $c$ such that, for some $\{x,y\} \in \bSCR_s$, $\phi'(x) = \phi'(y)=c$ and 
            \begin{itemize}
                \item  $c$ is not retain by at least one of $\{x,y\}$, that is $\{x,y\} \nsubseteq \dom(\phi^\star)$; or
                \item there exists $z\in N_s\setminus \{x,y\}$ such that $\phi'(z) =c$.
            \end{itemize}
            Observe that $Z_s = \Zsp - \Zsm$. We show that these auxiliary random variables are concentrated, and then deduce that $Z_s$ is concentrated as well.
    
            Given $\{x,y\}\in \bSCR_s$ and a color $c\in[\lceil\Delta/2\rceil]$, let $A^{\add}_{\{x,y\},c}$ be the event that $x$ and $y$ are assigned color $c$, and let $Z^{\add}_{\{x,y\},c}$ be the random variable equal to $1$ if $A^{\add}_{\{x,y\},c}$ holds, and $0$ otherwise.
            We have that 
            \[
            \Zsp \leq \sum_{\substack{\{x,y\}\in \bSCR_s,\\c\in [\lceil \Delta/2\rceil]}} Z^{\add}_{\{x,y\},c}.
            \]
            Therefore, by linearity of the expectation, 
            \begin{equation}
            \label{eq:EM^*}
            \EE(\Zsp) \leq \sum_{\substack{\{x,y\}\in \bSCR_s,\\c\in [\lceil \Delta/2\rceil]}} \PP(A^{\add}_{\{x,y\},c}) = \lceil\Delta/2\rceil \cdot |\bSCR_s| \cdot \left(\frac{1}{\lceil\Delta/2\rceil}\right)^2 \leq \frac{2}{\Delta}\cdot |\bSCR_s|.
            \end{equation}
            
            We note that changing the outcome of any color assignment affects $\Zsp$ by at most $1$. Moreover, whenever $\Zsp\geq k$ for some integer $k$, this can be certified by revealing at most $2k$ color assignments: for each color $c$ counted by $\Zsp$, it suffices to exhibit two vertices $x$ and $y$ with $\{x,y\}\in \bSCR_s$ and $\phi'(x) = \phi'(y)=c$. Therefore, by Lemma~\ref{lemma:talagrand} together with~\eqref{eq:EM^*}, we get
            \begin{equation}\label{eq:M^*}
                \PP(|\Zsp - \EE(\Zsp)|>\lambda)\leq 4\exp\left(\frac{-\lambda^2 \Delta}{128|\bSCR_s|+64\lambda\Delta}\right)
            \end{equation}
            for any real number $\lambda > 252\sqrt{|\bSCR_s|/\Delta}+688$.
    
            Similarly, $\Zsm$ is affected by at most $1$ when the outcome of any color assignment is changed.
            Moreover, whenever $\Zsm \geq k$, this can be certified by revealing at most $4k$ color assignments: for each one of the $k$ colors, it suffices to exhibit two vertices $x$ and $y$ with $\{x,y\}\in \bSCR_s$, and either
            \begin{itemize}
                \item one in-neighbor $x^-$ and one out-neighbor $x^+$ of $x$, or
                \item a vertex $z \in N_s \setminus \{x,y\}$,
            \end{itemize}
            such that all such vertices have been assigned that color. 
            By Lemma~\ref{lemma:talagrand},~\eqref{eq:EM^*}, and the fact that $\Zsm \leq \Zsp$ by definition, we have
            \begin{equation}\label{eq:M^del}
                \PP(|\Zsm-\EE(\Zsm)|>\lambda)\leq 4\exp\left(\frac{-\lambda^2 \Delta}{256|\bSCR_s|+128\lambda\Delta}\right)
            \end{equation}
            for any $\lambda> 357\sqrt{|\bSCR_s|/\Delta} + 1376$. 
            Recall that $Z_s = \Zsp-\Zsm$, hence by~\eqref{eq:M^*} and~\eqref{eq:M^del} we obtain
            \begin{align*}
            \PP(|Z_s-\EE(Z_s)|>\lambda)
            &\leq\PP\left(|\Zsp-\EE(\Zsp)|>\tfrac{\lambda}{2}\right) +\PP\left(|\Zsm-\EE(\Zsm)|>\tfrac{\lambda}{2}\right) \\
            &\leq 8\exp\left(\frac{-\lambda^2\Delta}{1024|\bSCR_s| + 256\lambda\Delta}\right)
            \end{align*}
            for any $\lambda> 714\sqrt{|\bSCR_s|/\Delta} + 2752$. The subclaim follows.
        \end{proofsubclaim}

        To make use of Subclaims~\ref{subclaim:sparse_1} and~\ref{subclaim:sparse_2}, we need to  estimate $|\bSCR_s|$. For this, let $m_s$ denote the number of arcs in $D^\star[N_s]$. Since there are at most two arcs between any two vertices, we have
        \begin{align*}
            |\bSCR_s| &\geq \textstyle \binom{d^+_{D^\star}(s)}{2} - \tfrac{1}{2} m_s \\
            &\geq \tfrac{1}{2}(\Delta-3b)^2 - \tfrac{1}{2}m_s &\text{by~\eqref{eq:mindeg_sparse},}\\
            &\geq \tfrac{1}{2}(\Delta-3b)^2 - \tfrac{1}{2}\Delta^2 + \tfrac{1}{6}d\Delta &\text{by Claim~\ref{claim:preserve_sparseness},}\\
            &= \tfrac{1}{6}d\Delta - 3b\Delta + \tfrac{9}{2}b^2\\
            &\geq \tfrac{1}{8}d\Delta  &\text{by~\eqref{eq:largeDelta:5}.}
        \end{align*}
        Together with Subclaim~\ref{subclaim:sparse_1}, we thus obtain 
        $\EE(Z_s) \geq \frac{1}{2^{20}}d$. In particular, $\frac{1}{2}\EE(Z_s) \geq 2b+1$ by~\eqref{eq:largeDelta:4}, which implies that
        \[
            \PP(Z_s < 2b+1) \leq \PP(|Z_s - \EE(Z_s)| > \tfrac{1}{2}\EE(Z_s)).
        \]
        Let thus $\lambda = \tfrac{1}{2}\EE(Z_s)$.
        Note that, by~\eqref{eq:largeDelta:5}, together with the fact that $|\bSCR_s| \geq \frac{1}{8}d\Delta$, we have
        \[
            2^{-16} \cdot\sqrt{|\bSCR_s|} \geq 715\sqrt{\Delta}.
        \]
        By multiplying each side of the inequality by $\sqrt{|\bSCR_s|}/\Delta$, we get that
        \begin{align*}
            \lambda  \geq 2^{-16} \cdot \frac{|\bSCR_s|}{\Delta} &\geq 715\cdot \sqrt{\frac{|\bSCR_s|}{\Delta}}
            \geq  714\cdot \sqrt{\frac{|\bSCR_s|}{\Delta}} + \sqrt{d/8}
            > 714\cdot \sqrt{\frac{|\bSCR_s|}{\Delta}} +2752,
        \end{align*}
        where the last inequality follows from~\eqref{eq:largeDelta:5}.
        Therefore, Subclaim~\ref{subclaim:sparse_2} can be applied with $\lambda = \frac{1}{2}\EE(Z_s)$, and we have
        \begin{align*}
            \PP(\text{$\phi$ is not $s$-extendable}) &\leq \PP(Z_s<2b+1)\\
            &\leq \PP(|Z_s-\EE(Z_s)| > \lambda)\\
            &\leq 8\exp\left(\frac{-\lambda\Delta}{1024\frac{|\bSCR_s|}{\lambda} + 256\Delta}\right) &\text{by Subclaim~\ref{subclaim:sparse_2},}\\
            &\leq 8\exp\left(\frac{-|\bSCR_s|}{2^{46}\Delta+ 2^{26}\Delta}\right) &\text{by Subclaim~\ref{subclaim:sparse_1},}\\
            &\leq 8\exp\left(-2^{-50}d\right) &\text{as $|\bSCR_s|\geq \tfrac{1}{8}d\Delta$,}\\
            &\leq \exp({-\log^2(\Delta)}) &\text{by~\eqref{eq:largeDelta:8}.}
        \end{align*}
        The claim follows.
    \end{proofclaim}

    \subsubsection{Probabilistic analysis: the dense sets}

    We now prove that, for every $i\in [t]$, $\phi$ is likely to be $i$-extendable.
    
    \begin{claim}
        \label{claim:prob_i_extendable}
        For every integer $i\in [t]$, $\PP(\text{$\phi$ is not $i$-extendable}) \leq \exp(- \log^2(\Delta))$.
    \end{claim}
    \begin{proofclaim}
        Let us fix $i\in [t]$. To prove the claim, we only have to prove that, with high probability, there exist at least $b$ loose $i$-rescuers (if $D_i$ is loose) or at least $b$ tight $i$-rescuers (if $D_i$ is tight).
        We first show that the expected number of $i$-rescuers is large, and then that this number is concentrated, hence showing that the number of $i$-rescuers is likely to be at least $b$.

        In order to avoid duplicating arguments between the loose and tight cases, let us provide a few more definitions. 
        First, if $D_i$ is loose, we let $K$ be $K_i$, otherwise we let $K$ be $K_i^\star$.
        Next, if $D_i$ is loose, let $\mathcal{Y}$ be the set of loose $i$-saviors, otherwise let $\mathcal{Y}$ be the set of tight $i$-saviors.

        Note that, in either case, $\mathcal{Y} \subseteq V(D) \cap V(D^\star)$, that is, the vertices in $\mathcal{Y}$ are not identified with another vertex when constructing $D^\star$. Indeed, if some $y\in \mathcal{Y}$ is identified with another vertex, it means that $y$ is spanned by $M^\star$. By definition, this happens only if $i\in \mathcal{I}$ and $y\in U_i^\star$. This is a contradiction, as if $i\in \mathcal{I}$, then $y$ is a tight $i$-savior, and by definition $y\in Y_i$.
        
        Finally, for every $y\in \mathcal{Y}$, let 
        \[
            \zSCR_y' = 
            \begin{cases}
                N^+_D(y) \cap \zSCR_i &\text{if $D_i$ is loose,}\\
                N^+_D(y) \cap \zSCR_i &\text{if $D_i$ is tight and $y$ is a tight $i^+$-savior,}\\
                N^-_D(y) \cap \zSCR_i &\text{if $D_i$ is tight and $y$ is a tight $i^-$-savior.}
            \end{cases}
        \]
        Further, for every $y\in \yCAL$, let $\zSCR_y$ be an arbitrary subset of $\zSCR_y'$ of size exactly
        \[
            \begin{cases}
                d^+(y)-\Delta+b & \text{if $D_i$ is loose,}\\
                d^+(y) - \Delta+b - \nu_i^\star  &\text{if $D_i$ is tight and $y$ is a tight $i^+$-savior,}\\
                d^-(y)-\Delta+b- \nu_i^\star &\text{if $D_i$ is tight and $y$ is a tight $i^-$-savior.}
            \end{cases}
        \]
        The existence of $\zSCR_y$ is guaranteed by the definition of $y$ being a loose or tight $i$-savior.
        In particular, note that $|\zSCR_y|\leq 2b+1$.
        Finally, for every $y\in \yCAL$, we let $\zSCR^\star_y$ be the set of vertices corresponding to $\zSCR_y$ in $D^\star$, that is
        \begin{align*}
            \zSCR_y^\star =~ &\{z \in \zSCR_y : \text{$z$ is not spanned by $M^\star$}\}\cup \{x_{z,u} : z\in \zSCR_y  \text{ and $z$ is matched to $u$ in $M^\star$}\}.
        \end{align*}

        Recall that, by definition, the vertices in $\zSCR_i$ can have only a few neighbors (namely, less than $2\log^4(\Delta))$) in $X_i$. It follows that we can extract a large set $\yCAL' \subseteq \yCAL$ for which the sets $(\zSCR_{y}^\star\cup \{y\})_{y\in \yCAL'}$ are pairwise disjoint.

    \begin{subclaim}
        There exists $\yCAL' \subseteq \yCAL$ such that $|\yCAL'| \geq \frac{\Delta}{24b \cdot \log^4(\Delta)}$ and the sets $(\zSCR_{y}^\star\cup \{y\})_{y\in \yCAL'}$ are pairwise disjoint.
    \end{subclaim}
    \begin{proofsubclaim}
        Let $\yCAL'\subseteq \yCAL$ be a set for which the sets $(\zSCR_{y}^\star\cup \{y\})_{y\in \yCAL'}$ are pairwise disjoint and, with respect to this property, has maximum cardinality.
        Let $\zSCR$ be the set of vertices $z\in \zSCR_i$ such that either
        \begin{itemize}
            \item $z\in \zSCR_y$ for some $y\in \yCAL'$, or
            \item $z$ is matched to a vertex $z^\star$ in $M^\star$ and $z^\star\in \zSCR_y$ for some $y\in \yCAL'$.
        \end{itemize} 
        By maximality of $\yCAL'$, in $D$, every vertex $y_1\in \yCAL \setminus \yCAL'$ has a neighbor $z_1\in \zSCR$. Moreover, every vertex in $\yCAL'$ has a neighbor in $\zSCR$ by definition. It follows that 
        \[
            \sum_{z\in \zSCR} |N_D(z)\cap X_i| \geq |\yCAL|.
        \]
        Moreover, since $\zSCR\subseteq \zSCR_i$, and by definition the vertices in $\zSCR_i$ have less than $2\log^4(\Delta)$ neighbors inside $X_i$ in $D$. Therefore,
        \[
            |\zSCR| \geq \frac{|\yCAL|}{2\log^4(\Delta)}.
        \]
        By definition, note that at least half of the vertices in $\zSCR$ belong to $\zSCR_y$ for some $y\in \yCAL'$. Therefore, it follows that
        \[
            \sum_{y\in \yCAL'} |\zSCR_y| \geq \frac{1}{2}|\zSCR|\geq \frac{|\yCAL|}{4\log^4(\Delta)}.
        \]
        Finally, by definition, for every $y\in \yCAL$ we have $|\zSCR_y|\leq 2b+1$. It follows that
        \[
            |\yCAL'|\cdot (2b+1)\cdot 4\log^4(\Delta) \geq |\yCAL|.
        \]
        Recall that, by Claims~\ref{claim:nb_loose_saviors} and~\ref{claim:nb_tight_saviors}, there exist at least $\tfrac{1}{2}\Delta$ loose $i$-saviors if $D_i$ is loose, and at least $\tfrac{1}{2}\Delta$ tight $i$-saviors if $D_i$ is tight. Therefore, $|\yCAL| \geq \tfrac{1}{2}\Delta$, and it follows from the inequality above that 
        \[
            |\yCAL'| \cdot 24b \cdot \log^4(\Delta) \geq\Delta.
        \]
        The subclaim follows.
    \end{proofsubclaim}
        
    Henceforth, let us thus fix $\yCAL' \subseteq \yCAL$ a subset of $\yCAL$ of cardinality precisely $\lceil \frac{\Delta}{24b \cdot \log^4(\Delta)}\rceil$, with the property that the sets $(\zSCR_{y}^\star\cup \{y\})_{y\in \yCAL'}$ are pairwise disjoint. Let $\tCAL=\bigcup_{y\in \yCAL'}(\zSCR_{y}^\star\cup \{y\})$ be the set of vertices spanned by these. 
   We finally let $T$ be the random variable counting the number of vertices $y\in\yCAL'$ such that:
    \begin{enumerateNum}
        \item\label{enum:good_triple_1} $y\notin \dom(\phi^\star)$;
        \item\label{enum:good_triple_2} $\zSCR_{y}^\star \subseteq \dom(\phi^\star)$; 
        \item\label{enum:good_triple_3} 
        for every $z\in \zSCR_y^\star$, there exists $x\in K$ such that $x\in \dom(\phi^\star)$ and $\phi'(z) = \phi'(x)$; and
        \item\label{enum:good_triple_4} for every vertex $z\in \zSCR_y^\star$ and every $u\in \tCAL\setminus \{z\}$, $\phi'(z) \neq \phi'(u)$.
    \end{enumerateNum}

    The key point here is that $T$ is a lower bound on the number of rescuers. The reader may note that condition~\ref{enum:good_triple_4} is not necessary for this, but we need it later in order to apply Azuma's inequality.

    \begin{subclaim}
        \label{subclaim:lb_rescuers}
        If $D_i$ is loose, then there exist at least $T$ loose $i$-rescuers. If $D_i$ is tight, then there exist at least $T$ tight $i$-rescuers. 
    \end{subclaim}
    \begin{proofsubclaim}
        Let us fix an arbitrary vertex $y\in \yCAL'$ such that $y$ satisfies conditions~\ref{enum:good_triple_1} to~\ref{enum:good_triple_4}, and let us show that $y$ is a loose (resp. tight) $i$-rescuer if $D_i$ is loose (resp. tight). 
        Observe first that, by~\ref{enum:good_triple_1}, $y$ is uncolored. Moreover, note that $y\in K$, since the vertices in $\yCAL$ are, by definition, loose (resp.\ tight) $i$-saviors, and every such vertex belongs to $K$.
        Then, note that, by conditions~\ref{enum:good_triple_2} and~\ref{enum:good_triple_3}, $\phi(\zSCR_y) \subseteq \phi(K)$.
        If $D_i$ is loose, it then follows that
        \[
            |\dom(\phi) \cap N^+_D(y)| - |\phi(N^+_D(y))| \geq |\zSCR_y| = d^+_D(y)-\Delta+b,
        \]
        hence showing that $y$ is a loose $i$-rescuer.
        Similarly, if $D_i$ is tight and $y$ is a tight $i^+$-savior, then $d^+(y)\leq \Delta$ and, since $K\cup \zSCR_y$ is disjoint from $U_i^\star$,
        \[
            |\dom(\phi) \cap N^+_D(y) \setminus U_i^\star|- |\phi(N^+_D(x) \setminus U_i^\star)| \geq d^+(y)-\Delta+b-\nu_i^\star,
        \]
        hence implying that $y$ is a tight $i$-rescuer. Finally, in the remaining case, that is $D_i$ is tight and $y$ is a tight $i^-$-savior, then $d^-(y)\leq \Delta$ and
        \[
            |\dom(\phi) \cap N^-(y) \setminus U_i^\star| - |\phi(N^-_D(x) \setminus U_i^\star)| \geq d^-(y)-\Delta+b-\nu_i^\star,
        \]
        again implying that $y$ is a tight $i$-rescuer.
    \end{proofsubclaim}

    It remains to show that $T$ is at least $b$ with probability at least $\exp(-\log^2(\Delta))$,
    and the claim then follows from Subclaim~\ref{subclaim:lb_rescuers}. As announced, we first prove that the expectation of $T$ is large, and then that it is concentrated.
    To show that $\EE(T)$ is large, we need a large set of vertices that will be candidates for being the vertices $x\in K$ of condition~\ref{enum:good_triple_3}. For this, we let
    \[
    \kCAL = K  \setminus \bigcup_{z\in \tCAL\setminus  \yCAL'} N_{D^\star}(z).
    \]
    Note that both $\yCAL'$ and $\kCAL$ are included in $K$, so vertices in $\yCAL'$ and $\kCAL$ are linked with digons. Note also that $\yCAL' \cap \kCAL = \emptyset$ by definition. Moreover, since the vertices in $\zSCR_i$ have at most $2\log^4(\Delta)$ neighbors in $X_i$, we have
    \begin{align*}
        |\kCAL| &\geq |K| - \sum_{y\in \yCAL'} \sum_{z\in \zSCR_y} |N_D(z) \cap X_i|\\
        &\geq |K| - |\yCAL'| \cdot (2b+1) \cdot 2\log^4(\Delta) &\text{by definition of $\zSCR_y$,}\\
        &\geq |K| - \tfrac{1}{4}\Delta &\text{by definition of $\yCAL'$,}\\
        &\geq \tfrac{3}{4}\Delta-\tfrac{4}{\epsilonb}d-8b &\text{by Claim~\ref{claim:size_K},}\\
        &\geq \tfrac{1}{2}\Delta &\text{by~\eqref{eq:largeDelta:3}.}
    \end{align*}
    We are now ready to estimate the expectation of $T$.

    \begin{subclaim}
    \label{subclaim:events_B_1}             
    $\EE(T)\geq \Delta^{2/3} + b$.
    \end{subclaim}
    \begin{proofsubclaim}
        For every vertex $y\in \yCAL$, we let $\ell(y) = |\zSCR_y^\star|$ and fix an arbitrary ordering $(z_1^y,\dots,z_{\ell(y)}^y)$ of $\zSCR_y^\star$. Given:
        \begin{itemize}
            \item a vertex $y\in \yCAL'$,
            \item $\ell(y)+1$ distinct vertices $\xCAL =(x_0,x_1,\dots,x_{\ell(y)})$ of $\kCAL$, and
            \item $\ell(y)+1$ distinct colors $\cCAL = (c_0,c_1,\dots,c_{\ell(y)})$ of $[\lceil\Delta/2\rceil]$,
        \end{itemize} 
        we let $A_{y,\xCAL,\cCAL}$ be the event that
        \begin{itemize}
            \item $\phi'(y) = \phi'(x_0) = c_0$ and $\phi'(u) \neq c_0$ for every other vertex $u$ in $\kCAL \cup \tCAL$; and
            \item for every $1\leq j \leq \ell(y)$, $\phi'(z_j^y) = \phi'(x_j) = c_j$, and $\phi'(u)\neq c_j$ for every other vertex $u$ in $\kCAL \cup \tCAL \cup N^+_{D^\star}(z_j^y)\cup N^+_{D^\star}(x_j)$.
        \end{itemize}
        Since $\{yx_0,x_0y\}$ is a digon, observe that $y$ must be counted by $T$ when $A_{y,\xCAL,\cCAL}$ holds. 
        Since, in $D^\star$, every vertex has out-degree at most $\Delta+31b$ by Claim~\ref{claim:degrees_Ds}, we have
        \[
        \PP(A_{y,\xCAL,\cCAL}) \geq\left(\frac{1}{\Delta}\right)^{2\ell(y)+2} \cdot \left(1-(\ell(y)+1)\cdot \frac{2}{\Delta}\right)^{|\kCAL|+|\tCAL|+\ell(y)\cdot (\Delta+31b)}. 
        \]
        Recall that $\ell(y) \leq 2b+1$. Therefore, using that $|\kCAL| \leq \Delta$ (as $\kCAL$ is a biclique), that 
        \[
            |\tCAL| \leq (2b+1) \cdot |\yCAL| \leq \Delta / 8\log^4(\Delta) \leq \Delta,
        \]
        and that $\Delta+31b \leq 2\Delta$, we obtain that 
        \[
        \PP(A_{y,\xCAL,\cCAL}) \geq\left(\frac{1}{\Delta}\right)^{2\ell(y)+2} \cdot \left(1- \frac{6b}{\Delta}\right)^{8b\cdot \Delta}\geq  \left(\frac{1}{\Delta}\right)^{2\ell(y)+2} \cdot \frac{1}{2^{96b^2}}
        \]
        by~\eqref{eq:largeDelta:6}.
        For every $y\in \yCAL'$, since $|\kCAL| \geq \tfrac{1}{2}\Delta$, note that we have at least
        \[
            \tfrac{1}{2}\Delta \cdot (\tfrac{1}{2}\Delta-1) \cdot \ldots \cdot (\tfrac{1}{2}\Delta-\ell(y)) \geq (\tfrac{1}{4}\Delta)^{\ell(y)+1}
        \]
        choices for $\xCAL$, and at least 
        \[
            \tfrac{1}{2}\Delta \cdot (\tfrac{1}{2}\Delta-1) \cdot \ldots \cdot (\tfrac{1}{2}\Delta-\ell(y)) \geq (\tfrac{1}{4}\Delta)^{\ell(y)+1}
        \]
        choices for $\cCAL$, where in both cases we use that $\frac{1}{2}\Delta-\ell(y) \geq \frac{1}{2}\Delta-2b-1 \geq \frac{1}{4}\Delta$ by~\eqref{eq:largeDelta:3}. For every $y\in \yCAL'$, let $A_y$ be the even that $y$ is counted by $T$. For every fixed $y\in\yCAL'$, since the events of the form $A_{y,\xCAL,\cCAL}$ are pairwise disjoint, we have
        \[
            \PP(A_y) \geq \sum_{\mathcal{X},\mathcal{C}} \PP(A_{y,\xCAL,\cCAL}) \geq (\tfrac{1}{4}\Delta)^{2\ell(y)+2} \cdot (\tfrac{1}{\Delta})^{2\ell(y)+2} \cdot \frac{1}{2^{96b^2}} \leq 2^{-96b^2-8b-8},
        \]
        where in the equality we used that $\ell(y)\leq 2b+1$.
        Hence, by linearity of the expectation, we have 
        \begin{align*}
        \EE(T) = \sum_{y\in \yCAL'} \EE(Z_y)
            \geq |\yCAL'| \cdot 2^{-96b^2 - 8b-8}&\geq 2^{-96b^2 - 8b-12}\cdot \frac{\Delta}{b\log^4 \Delta} &\text{by definition of $\yCAL'$,}\\
            &\geq \Delta^{2/3}  + b&\text{by~\eqref{eq:largeDelta:9}.}
        \end{align*}
    \end{proofsubclaim}

    We now show that $T$ is concentrated.
    
    \begin{subclaim}
        \label{subclaim:events_B_2} 
        For any $\lambda\geq 0$, $\PP(|T-\EE(T)|>\lambda)\leq 2\exp\left(\frac{-\lambda^2}{144\Delta}\right)$.
    \end{subclaim}
    \begin{proofsubclaim}
        We aim to apply Azuma's inequality. Let us fix a labeling $w_1,\ldots,w_n$ of the vertices of $D$ such that, for some $q$, $w_1,\ldots,w_q\in V(D)\setminus(\kCAL \cup \tCAL)$, and $w_{q+1},\ldots w_n\in \kCAL \cup \tCAL$. 
        Note that $T$ is determined by the $n=|V(D)|$ assignments $\phi'(w_1),\dots,\phi'(w_n)$. 
        
        Let $(c_1,\dots,c_n)$ and $(c_1',\dots,c_n')$ be two arbitrary sequences of $n$ colors of $[\lceil \Delta/2\rceil ]$.
        In order to apply Azuma's inequality, let us bound, for every $j\in [n]$, the value of
        \begin{align*}
            \delta_j = ~&\Big| \EE\big(T \mid \phi'(w_1) = c_1 \cap \dots\cap  \phi'(w_{j-1}) = c_{j-1} \cap \phi'(w_{j}) = c_{j}\big)\\ 
            &- \EE\big(T \mid \phi'(w_1) = c_1\cap \dots\cap  \phi'(w_{j-1}) = c_{j-1} \cap \phi'(w_{j}) = c_{j}'\big) \Big|.
        \end{align*}

        For better readability, for every $j\in [n]$, let $W_j$ denote the event $\phi'(w_1) = c_1 \cap \ldots \cap \phi'(w_j) = c_j$, and let $W_j'$ denote the event $\phi'(w_1) = c_1 \cap \ldots \cap \phi'(w_{j-1}) = c_{j-1} \cap \phi'(w_j) = c_j'$, so
        \[
            \delta_j = \Big|\EE\big(T\mid W_{j}\big) - \EE\big(T\mid W_j'\big)\Big|.
        \]

        First, observe that changing the assigned color $\phi'(w_j)$ of a single vertex $w_j$ from $c_j$ to $c_j'$ can affect the value of $T$ by at most $2$, as by Condition~\ref{enum:good_triple_4} at most one vertex of $\tCAL$ is colored with $c_j$, and at most one is colored with $c_j'$. Therefore, for every $j\in [n]$, we have $\delta_j \leq 2$. In particular, it follows that
        \begin{equation}
            \label{eq:delta_j}
            \sum_{j=q+1}^n \delta_j^2 \leq 4\cdot |\kCAL \cup \tCAL| \leq 8\Delta.
        \end{equation}

        Assume now that $j\leq q$ and let us show a more precise bound in this case (that is, when $w_1,\dots,w_j \notin \kCAL\cup \tCAL$). Let $N_j = N_{D^\star}(w_j) \cap \tCAL$, and let $F_j$ be the event that some vertex in $N_j$ receives color $c_j$ or $c_j'$ through $\phi'$. When $F_j$ does not hold, observe that changing the color of $w_j$ from $c_j$ to $c_{j}'$ does not affect the value of $T$, so in particular
        \begin{equation}
            \label{eq:equality_EE}
            \EE(T\mid W_j \cap \overline{F_j}) = \EE(T\mid W_j' \cap \overline{F_j}).
        \end{equation}
        Moreover, since $j\leq q$, we have that $\{w_1,\dots,w_j\}$ is disjoint from $N_j$. This implies 
        \begin{equation}
        	\label{eq:prob_FJ_bar}
        	\PP(\overline{F_j}) = \PP(\overline{F_j} \mid W_j)= \PP(\overline{F_j} \mid W_j'),
        \end{equation}
        and
        \begin{equation}
            \label{eq:prob_FJ}
           \PP(F_j) =  \PP(F_j \mid W_j)= \PP(F_j \mid W_j')   \leq |N_j| \cdot \frac{4}{\Delta},
        \end{equation}
        where the last inequality follows from the union bound and the fact that some vertex receives $c_j$ or $c_j'$ with probability $\frac{2}{\lceil \Delta/2\rceil} \leq 4/\Delta$.
        Combining the inequalities above, we obtain that
        \begin{align*}
            \delta_j &= \Big|\EE\big(T\mid W_{j}\big) - \EE\big(T\mid W_j'\big)\Big|\\
            &= \Big| \EE(T\mid W_j \cap F_j) \cdot \PP(F_j\mid W_j) + \EE(T\mid W_j \cap \overline{F_j}) \cdot \PP(\overline{F_j}\mid W_j) \\
            & \hspace*{0.5cm}-\EE(T\mid W_j' \cap F_j) \cdot \PP(F_j\mid W_j') - \EE(T\mid W_j' \cap \overline{F_j}) \cdot \PP(\overline{F_j} \mid W_j') \Big| \\
            &= \Big| \EE(T\mid W_j \cap F_j) \cdot \PP(F_j\mid W_j) -\EE(T\mid W_j' \cap F_j) \cdot \PP(F_j\mid W_j')\Big| &\text{by~\eqref{eq:equality_EE} and~\eqref{eq:prob_FJ_bar},}\\
            &\leq \Big| \EE(T\mid W_j \cap F_j) -\EE(T\mid W_j' \cap F_j) \Big|\cdot \frac{4\cdot |N_j|}{\Delta} &\text{by~\eqref{eq:prob_FJ}.}
        \end{align*}
        Since changing the color of a single vertex affects $T$ by at most $2$, it follows that, for every $j\leq q$,
        \begin{equation}
            \label{eq:delta_j_2}
            \delta_j \leq \frac{8|N_j|}{\Delta}.
        \end{equation}
        We thus have 
        \begin{align*}
            \sum_{j=1}^n \delta_j^2 &= \sum_{j=1}^q \delta_j^2 + \sum_{j=q+1}^n \delta_j^2\\
            &\leq 2\sum_{j=1}^q \delta_j + 8\Delta &\text{by~\eqref{eq:delta_j} and the fact that $\delta_j\leq 2$,}\\
            &\leq 8\Delta + \frac{16}{\Delta}\sum_{j=1}^q |N_j| &\text{by~\eqref{eq:delta_j_2}.}
        \end{align*}
        Recall that, by definition of $N_j$, $\sum_{j=1}^q |N_j|$ is in particular bounded by the number of arcs of $D^\star$ having an extremity in $\tCAL$, which is at most $2\cdot \delmax(D^\star) \cdot |\tCAL|$. By Claim~\ref{claim:degrees_Ds} and the facts that $\Delta+31b\leq 2\Delta$ and that $|\tCAL|\leq \Delta$, it then follows that
        \[
            \sum_{j=1}^n \delta_j^2 \leq 8\Delta + \frac{16}{\Delta} \cdot 2 \cdot  \delmax(D^\star) \cdot |\tCAL| \leq 72\Delta.
        \]
        By Lemma~\ref{lemma:azuma}, it follows that, for every $\lambda \geq 0$,
        \[
            \PP(|T-\EE(T)| > \lambda) \leq 2\exp( - \lambda^2 / 144\Delta),
        \]
        as desired. The subclaim follows.
    \end{proofsubclaim}

    We are now ready to combine the different subclaims and to conclude the proof of Claim~\ref{claim:prob_i_extendable}.
    We have
    \begin{align*}
        \PP(\text{$\phi$ is not $i$-extendable}) &\leq \PP(T < b) &\text{by Subclaim~\ref{subclaim:lb_rescuers},}\\
        &\leq \PP\big(|T-\EE(T)| > \Delta^{2/3}\big) &\text{by Subclaim~\ref{subclaim:events_B_1},}\\
        &\leq 2\exp\left(-\Delta^{1/3}/144\right) &\text{by Subclaim~\ref{subclaim:events_B_2},}\\
        &\leq \exp(-\log^{2}(\Delta)) &\text{by~\eqref{eq:largeDelta:8},}
    \end{align*}
    as desired. The claim follows.
    \end{proofclaim}

    \subsubsection{Probabilistic analysis: conluding with Lov\'asz Local Lemma}

    For every $s\in S$, let $A_s$ be the event that $\phi$ is not $s$-extendable, and for every $i\in [t]$ let $A_i$ be the event that $\phi$ is not $i$-extendable. By Claims~\ref{claim:prob_s_extendable} and~\ref{claim:prob_i_extendable}, each of these bad events holds with probability at most 
    \[
        p = \exp(-\log^2(\Delta)).
    \]

    Note that, for every vertex $u\in V(D)$, whether $u\in \dom(\phi)$ depends only on the colors assigned by $\phi'$ to the vertex corresponding to $u$ in and its neighbors in $D^\star$. 
    
    Therefore, for any vertex $s\in S$, whether $\phi$ is $s$-extendable may depend only on the colors $\phi'(u)$ given to each vertex $u$ at distance at most $2$ from $s$ in the underlying graph of $D^\star$.
    Recall that, by Claim~\ref{claim:adjacency_Uis_bis}, whenever two vertices are identified, they have at least one neighbor in common. Therefore, whether $\phi$ is $s$-extendable may depend only on the colors given to vertices in $D^\star$ corresponding to vertices in $D$ which are at distance at most $6$ from $s$.

    Similarly, for every $i\in [t]$, whether $\phi$ is $i$-extendable may depend only on the colors given to vertices in $D^\star$ corresponding to vertices in $D$ which are at distance at most $6$ from some vertex in $X_i$.

    It follows that, for every vertex $s\in S$, $A_s$ is mutually independent from all events $A_{s'}$ with $s$ and $s'$ being at distance at least $13$ from each other, and from all events $A_{i}$ with $s$ being at distance at least $13$ from all vertices in $X_i$.
    Hence, since the underlying graph of $D$ has maximum degree at most $2(\Delta+b+1)$ by Claim~\ref{claim:degrees}, $A_s$ is mutually independent from all other events except at most $12\cdot (2\Delta+2b+2)^{12}$ of them.
    
    Similarly, for every $i\in [t]$, $A_i$ is mutually independent from all events $A_{s}$ with $s$ being at distance at least $13$ from $X_i$, and from all events $A_j$ with $X_i$ being at distance at least $13$ from $X_j$. Since $|X_i| \leq \Delta+5d \leq 2\Delta$ by Claim~\ref{claim:dense_decomposition}\ref{claim:dense_decomposition:i} and by~\eqref{eq:largeDelta:3},
    $A_i$ is mutually independent from all events but at most $(2\Delta)\cdot 12\cdot (2\Delta+2b+2)^{12}$ of them.

    Since
    \[
        4 \cdot \exp(-\log^2(\Delta)) \cdot (2\Delta+2b+2)^{13} \leq 1
    \]
    by~\eqref{eq:largeDelta:10}, it follows from Lov\'asz Local Lemma (Lemma~\ref{lemma:lll}) that $\phi$ is extendable with positive probability. By Claim~\ref{claim:no_extendable_dicoloring}, we get that $\dic(D) \leq \Delta-b+1$, a contradiction. Theorem~\ref{thm:main_reform} follows.
\end{proof}

\bibliographystyle{abbrv}
\bibliography{refs}

\end{document}